\documentclass{amsart}
\usepackage{amssymb}
\usepackage{amscd}
\usepackage{amsmath}
\usepackage{amsfonts}
\usepackage{latexsym}
\usepackage[all]{xy}
\usepackage{verbatim}
\newtheorem{theorem}{Theorem}[section]
\newtheorem{lemma}[theorem]{Lemma}
\newtheorem{proposition}[theorem]{Proposition}
\newtheorem{corollary}[theorem]{Corollary} 
\theoremstyle{definition}  
\newtheorem{definition}[theorem]{Definition}
\newtheorem{example}[theorem]{Example}
\newtheorem{conjecture}[theorem]{Conjecture}  

\newtheorem{remark}[theorem]{Remark}

\newcommand{\Tr}{\text{Tr}}
\newcommand{\id}{\text{id}} 
 
\newcommand{\Ker}{\text{Ker\,}} 

\newcommand{\Fun}{\text{Fun}}
\newcommand{\FPdim}{\text{FPdim}} 

\renewcommand{\Vec}{\operatorname{\operatorname{\mathsf{Vec}}}}
\DeclareMathOperator{\Pic}{\operatorname{\mathsf{Pic}}}
\DeclareMathOperator{\BrPic}{\operatorname{\mathsf{BrPic}}}

\DeclareMathOperator{\Aut}{\operatorname{\mathsf{Aut}}}
\DeclareMathOperator{\Rad}{\operatorname{\mathsf{Rad}}}
\DeclareMathOperator{\Stab}{\operatorname{\mathsf{Stab}}}
\DeclareMathOperator{\OutStab}{\operatorname{\mathsf{OutStab}}}

\DeclareMathOperator{\Out}{\operatorname{\mathsf{Out}}}
\DeclareMathOperator{\Inn}{\operatorname{\mathsf{Inn}}}
\DeclareMathOperator{\Inv}{\operatorname{\mathsf{Inv}}}

\DeclareMathOperator{\Rep}{\operatorname{\mathsf{Rep}}}
\DeclareMathOperator{\Hom}{\operatorname{\mathsf{Hom}}}

\DeclareMathOperator{\Sym}{\operatorname{\mathsf{S}}}

\newcommand{\ad}{\text{ad}}

\newcommand{\B}{\mathcal{B}}
\newcommand{\C}{\mathcal{C}}
\newcommand{\D}{\mathcal{D}}

\newcommand{\Z}{\mathcal{Z}}

\renewcommand{\L}{\mathcal{L}}
\newcommand{\M}{\mathcal{M}}
\newcommand{\A}{\mathcal{A}}

\newcommand{\KP}{\mathcal{K}\mathcal{P}}
\newcommand{\TY}{\mathcal{T}\mathcal{Y}}

\newcommand{\be}{\mathbf{1}}

\newcommand{\g}{\mathfrak{g}}

\DeclareMathOperator{\ind}{\operatorname{\mathsf{ind}}}
\DeclareMathOperator{\conj}{\operatorname{\mathsf{conj}}}

\renewcommand{\be}{\mathbf{1}}

\newcommand{\bt}{\boxtimes}
\newcommand{\ot}{\otimes}
\renewcommand{\g}{\mathfrak{g}}

\newcommand{\beq}{\begin{equation}}
\newcommand{\eeq}{\end{equation}}

\newcommand{\bpf}{\begin{proof}}
\newcommand{\epf}{\end{proof}}

\newcommand{\bth}{\begin{theorem}}
\renewcommand{\eth}{\end{theorem}}
\newcommand{\bpr}{\begin{proposition}}
\newcommand{\epr}{\end{proposition}}
\newcommand{\ble}{\begin{lemma}}
\newcommand{\ele}{\end{lemma}}
\newcommand{\bco}{\begin{corollary}}
\newcommand{\eco}{\end{corollary}}
\newcommand{\bde}{\begin{definition}}
\newcommand{\ede}{\end{definition}}
\newcommand{\bex}{\begin{example}}
\newcommand{\eex}{\end{example}}
\newcommand{\bre}{\begin{remark}}
\newcommand{\ere}{\end{remark}}
\newcommand{\bcj}{\begin{conjecture}}
\newcommand{\ecj}{\end{conjecture}}

\begin{document}
\title[Brauer-Picard groups of fusion categories]
{On the Brauer-Picard groups of fusion categories}
\author{Ian Marshall}
\address{Department of Mathematics and Statistics,
University of New Hampshire,  Durham, NH 03824, USA}
\email{isg3@wildcats.unh.edu}
\author{Dmitri Nikshych}
\address{Department of Mathematics and Statistics,
University of New Hampshire,  Durham, NH 03824, USA}
\email{dmitri.nikshych@unh.edu}

\begin{abstract}
We develop methods of computation of the Brauer-Picard groups of fusion categories and apply them
to compute such groups for several classes of fusion categories of prime power dimension: representation
categories of elementary abelian groups with twisted associativity,  extra special $p$-groups, and the Kac-Paljutkin 
Hopf algebra. We conclude that many finite groups of Lie type occur as  composition factors of the Brauer-Picard groups
of pointed fusion categories. 
\end{abstract}

\maketitle


\section{Introduction}

Let $\A$ be a fusion category.  The Brauer-Picard group $\BrPic(\A)$ was introduced in \cite{ENO1} as
the group of equivalence classes of invertible $\A$-bimodule categories. The motivation for this  definition
is that given a finite group $G$ and a $G$-graded fusion category $\B =\oplus_{g\in G}\,\B_g$ whose trivial
component $\B_1$ is equivalent to $\A$ there is a group homomorphism 
\begin{equation}
\label{G to BrPic}
G\to \BrPic(\A)
\end{equation}
and, conversely, any such $\B$ can be constructed from  a homomorphism \eqref{G to BrPic} and some cohomological data.
It was also shown in \cite{ENO1}  that there is a canonical isomorphism
\begin{equation}
\label{BP=Aut}
\BrPic(\A) \cong \Aut^{br}(\Z(\A)),
\end{equation}
where $\Aut^{br}(\Z(\A))$ is the group of isomorphism classes of braided tensor autoequivalences of the center $\Z(\A)$ of $\A$. 
The latter group can be thought of as a categorical analogue of an orthogonal group. For example, if $\A$ is the representation
category of a finite abelian group $A$, one has $\Aut^{br}(\Z(\A)) \cong O(A \oplus \widehat{A})$, where $\widehat{A}$ is the group of
characters of $A$ and $O(A \oplus \widehat{A})$ is the group of automorphisms of $A \oplus \widehat{A}$ preserving its canonical
quadratic form. 

The definition of $\BrPic(\A)$ is rather abstract, while the group $\Aut^{br}(\Z(\A))$ is defined in more concrete terms
and, hence, is easier to compute.  There is also an important geometric intuition coming from the action of $\Aut^{br}(\Z(\A))$ 
on Lagrangian subcategories of the center (these are categorical counterparts  of Lagrangian subspaces).  
Such actions were used in \cite{BN, NR, R}
to compute Brauer-Picard groups of various symmetric tensor categories.  An analogue of the Bruhat
decomposition of $\BrPic(\Rep(G))$, where $G$ is a finite group, was recently discussed in \cite{LP}. 

Let $p$ be a prime. By a {\em fusion $p$-category} we mean a  fusion category whose Frobenius-Perron dimension is
a power of $p$.  For odd $p$ such a category is always integral, i.e., is equivalent to the representation category of some
semisimple quasi-Hopf algebra \cite{GN}, while for $p=2$  it is a $\mathbb{Z}_2$-extension of an integral fusion category. 

Fusion $p$-categories were classified 
in \cite{DGNO1}  (see also \cite[Section 9.14]{EGNO}).  Namely, it was shown that any integral fusion $p$-category $\A$ is 
{\em group-theoretical}, i.e., it is categorically Morita equivalent to a pointed fusion category.  
Explicitly, this means that $\A$ is equivalent to the category of bimodules
over an algebra in  the category $\C(G,\, \omega)$ of vector spaces graded by a $p$-group $G$
with the associativity constraint given by $\omega\in H^3(G,\, k^\times)$.  Thus, all fusion $p$-categories
can be described  in terms of $p$-groups and their cohomology. 

In this paper we develop methods of computing the group $\BrPic(\A)$, where $\A$ is a fusion category,
using the induction homomorphism
\begin{equation}
\label{intro ind}
\ind: \Aut(\A)\to \Aut^{br}(\Z(\A)) \cong  \BrPic(\A)
\end{equation}
from the group of tensor autoequivalences of $\A$ to the group of braided autoequivalences of its center,
see Section~\ref{section 3} for a precise definition. The image of this homomorphism is a subgroup
of $\BrPic(\A)$ which may or may not coincide with the whole group.  If $\tilde\A$ is a fusion category 
categorically Morita equivalent to $\A$ (so that  $\BrPic(\A)\cong \Aut^{br}(\Z(\tilde\A))$  then one can  also
induce central autoequivalences of $\Z(\A)$ from $\tilde{\A}$. Thus, in general, there are several
induction homomorphisms that can be combined to compute the Brauer-Picard group. 
Another useful tool of constructing elements of the Brauer-Picard group is the Picard induction of an invertible $\Z(\A)$-module category 
(i.e., a $\A$-bimodule category)  from an invertible category over a subcategory $\D\subset \Z(\A)$.

We use these methods to begin a systematic study of Brauer-Picard groups of  fusion $p$-categories. 
These groups turn out to be rather large and interesting. In particular, they include extensions of 
many classical finite groups of Lie type.
Since the Brauer-Picard group is invariant under categorical Morita  equivalence, it suffices to compute 
the Brauer-Picar groups of categories $\A= \C(G,\, \omega)$, where $G$ is a $p$-group. We explicitly compute
these groups in the cases  when $G$ is an elementary abelian or extra special $p$-group.  
We also compute the Brauer-Picard group of the 
representation category of the celebrated Kac-Paljutkin Hopf algebra~\cite{KP}.


The paper is organized as follows. 

In Section~\ref{section 2} we recall definitions and basic facts about
Brauer-Picard groups and tensor autoequivalences of fusion categories. 

In Section~\ref{section 3}  we study the induction homomorphism \eqref{intro ind}. 
Proposition~\ref{sse for ker ind} gives an exact sequence that can be used to compute the kernel of $\ind$.  
We characterize the image of induction in Theorem~\ref{alphaA=A} as the stabilizer of a certain Lagrangian algebra
in the center. The orbits of the action of the Brauer-Picard group on the categorical Lagrangian Grassmannian are described
in Section~\ref{section 3.4}. 

In Section~\ref{section 4} we prove several results about the  structure of representation categories of 
twisted quantum doubles (i.e., categories $\Z(\C(G,\, \omega))$). In particular, we analyze the image of 
the above induction homomorphism in this case and give a sufficient conditions for it to be surjective
(Corollary~\ref{Indsurj}).  We also note a useful braided equivalence \eqref{G and AK}
that allows to replace the group $G$ by an easier group.

In Section~\ref{section 5} for an  elementary abelian $p$-group  $V$ we  describe the cohomology group $H^3(V,\, k^\times)$ 
as a $GL(V)$-module and obtain exact sequences containing  the Brauer-Picard groups of categories $\C(V,\, \omega)$.
We show that under certain non-degeneracy conditions these Brauer-Picard groups are abelian extensions of stabilizers
of certain elements in exterior and symmetric algebras of $V^*$. 

Finally, in Section~\ref{section 6} we compute the Brauer-Picard groups of several classes of 
fusion $p$-categories: representation categories of extra special $p$-groups, pointed fusion categories
coming from simple modular Lie algebras, and the representation category of the Kac-Paljutkin 
Hopf algebra of dimension $8$.

\textbf{Acknowledgments} We are grateful to Derek Holt, Geoffrey Mason, Viktor Ostrik, and Leonid Vainerman
for helpful discussions.  We thank  Cesar Galindo for useful discussions and for pointing the paper \cite{GP} to us. 
We thank Ehud Meir for explaining to us the classification of module categories
over Tambara-Yamagami categories from \cite{MM} and sharing his calculations with us. The authors were partially 
supported by the NSA grant H98230-15-1-0003. The second author was partially supported by the NSA grant H98230-16-1-0008.


\section{Preliminaries}
\label{section 2}

\subsection{Fusion categories and Brauer-Picard groups}

In this paper we work over an algebraically closed field $k$  of characteristic $0$.
By a {\em fusion category} we mean a semisimple rigid tensor category with finitely many
isomorphism classes of simple objects and finite dimensional spaces of morphisms. 
For basic results of the theory of fusion categories see \cite{EGNO, DGNO2}.

For a fusion category $\A$  let $\Aut(\A)$ denote the group of isomorphism classes of tensor autoequivalences
of $\A$.  If $\A$ is a braided fusion category we denote $\Aut^{br}(\A)$ the group of isomorphism classes of 
braided tensor autoequivalences of $\A$.

For a group $G$ let $\Aut(G)$ denote the group of automorphisms of $G$, let $\Inn(G)$ denote the subgroup 
of inner automorphisms, and let $\Out(G)=\Aut(G)/\Inn(G)$ be the group of (classes of) outer automorphisms of $G$. 

A fusion category is called {\em pointed} if all its simple objects are invertible with respect to the tensor product.
For a fusion category $\A$ let $\A_{pt}$ denote its maximal pointed fusion subcategory, i.e., the fusion
subcategory generated by invertible simple objects of $\A$. 

Any pointed fusion category is equivalent to some category $\C(G,\, \omega)$ of vector spaces graded by a finite group $G$
with the associativity constraint given by a $3$-cocycle $\omega\in H^3(G,\, k^\times)$. 
The center $\Z(\C(G,\, \omega))$ is known to be equivalent
to the representation category of the  twisted group double, i.e., the quasi-Hopf algebra $D^\omega(G)$.

\begin{remark}
A common notation for $\C(G,\, \omega)$ in the literature is $\Vec_G^\omega$. We choose a different notation
in order to keep subscripts reasonable. 
\end{remark}

The {\em Brauer-Picard} group $\BrPic(\A)$ of a fusion category $\A$ consists of equivalence classes of invertible $\A$-bimodule categories 
\cite{ENO1} with multiplication given by the tensor product over $\A$. 
It was shown in \cite{ENO1, ENO2} that there is a canonical isomorphism
\begin{equation}
\label{Phi}
\Phi: \BrPic(\A)\xrightarrow{\sim} \Aut^{br}(\Z(\A)),
\end{equation}
where $\Z(\A)$ denotes the center of $\A$. 

More generally, for a braided fusion category $\C$ the {\em Picard} group of $\C$ consists of equivalence classes of 
one-sided $\C$-module categories.  When $\C$ is non-degenerate, there is an isomorphism
\begin{equation}
\label{PhiC}
\Phi: \Pic(\C) \xrightarrow{\sim} \ \Aut^{br}(\C).
\end{equation}
Note that  \eqref{Phi} is a special case of \eqref{PhiC} since $\Pic(\Z(\A)) \cong \BrPic(\A)$. 
 
Isomorphism \eqref{Phi} was used in \cite{NR} to analyze Brauer-Picard groups of categories $\C(G,\, 1)=\Vec_G$ and compute
them in concrete examples. 

\subsection{Fusion $p$-categories}

Let $p$ be a prime integer.  By a {\em fusion $p$-category} we mean a fusion category whose Frobenius-Perron dimension
is $p^n$ for some integer~$n$. Such categories were characterized in \cite{DGNO1} (see also \cite[Section 9.4]{EGNO}).
Namely, any such category $\A$  which is {\em integral} (i.e., such that  every object of $\A$ has an 
integral Frobenius-Perron dimension) is group-theoretical,
i.e., there is a $p$-group $G$ and $\omega\in H^3(G,\, k^\times)$ such that $\A$ is categorically Morita equivalent to
 $\C(G,\,\omega)$.  In particular, 
 \[
 \BrPic(\A) \cong \BrPic(\C(G,\,\omega)).
 \] 

\begin{remark}
When $p$ is odd,  a fusion $p$-category is automatically integral. 
\end{remark}

Computations of  Brauer-Picard groups of concrete ``small" examples of fusion $p$-categories were given in \cite{NR}.
In particular, it was shown that
\begin{equation}
\label{BrPic d8 and q8}
\BrPic(\C(D,\,1)) = S_4,\qquad \BrPic(\C(Q,\,1)) = S_3,
\end{equation}
where $D$ and $Q$ are, respectively,  the dihedral and quaternion groups of order $8$  and $S_n$
denotes the symmetric group of degree $n$. 

Furthermore, it was shown in \cite{R}, that if $G$ is  the (unique) non-abelian group of order $p^3$ and exponent $p$,
where $p$ is odd, then 
\begin{equation}
\label{BrPic d of order p3}
\BrPic(\C(G,\,1)) = SL_3(\mathbb{F}_p). 
\end{equation}

\subsection{Tensor autoequivalences of pointed fusion categories}
\label{tens aut ptd}

Let $G$ be a finite group and let $\omega\in H^3(G,\, k^\times)$ be a $3$-cocycle. 

Let $\Stab(\omega) \subset \Aut(G)$ be the subgroup of automorphisms $a\in \Aut(G)$ such that $\omega\circ(a\times a\times a)$
and $\omega$ are cohomologous.  Note that $\Inn(G)\subset \Stab(\omega)$. 
The following result is well known (and is easy to check).

\begin{proposition}
There is a short exact sequence
\begin{equation}
\label{AutCGw}
0\to H^2(G,\, k^\times) \to \Aut(\C(G,\, \omega)) \to \Stab(\omega) \to 0.
\end{equation} 
\end{proposition}
Here $H^2(G,\, k^\times)$ parameterizes the tensor functor structures on the identity endofunctor
of $\C(G,\, \omega)$. Given $a\in \Stab(\omega)$ and a $2$-cochain $\mu$ such that  
\[
d^2(\mu) =\frac{\omega\circ(a\times a\times a)}{\omega}
\]
let $F_{a,\mu}$ denote the corresponding autoequivalence of $\C(G,\, \omega)$.

\subsection{Lagrangian subcategories}

Let $\C$ be a non-degenerate braided fusion category.
A fusion subcategory $\mathcal{L}\subset \C$ is called {\em Lagrangian} if $\L$ is Tannakian,
i.e., $\L =\Rep(G)$ for some finite group $G$, and  $\L$ coincides with its M\"uger centralizer, i.e., $\L =\L'$.  
In this case the regular algebra $A =\Fun(G)$ of $\Rep(G)$ is a Lagrangian algebra in $\C$  (see Section~\ref{section 3.3})
and  the category $\C_A$ of left $A$-modules in $\C$ is equivalent to
$\C(G,\, \omega)$ for some $\omega\in H^3(G,\,k^\times)$, so that there is a braided tensor 
equivalence 
\begin{equation}
\label{C=ZVecGw}
\C \cong \Z(\C(G,\, \omega)).
\end{equation}
See \cite[Section 4.4.10]{DGNO2} for details.

\subsection{Properties of the Picard induction}
\label{section 2.5}

Let $\C$ be a braided fusion category and let $\D\subset \C$ be a fusion subcategory.
There is a homomorphism of Picard groups given by the induction:
\begin{equation}
\label{Picard induction}
P_\D^\C: \Pic(\D) \to \Pic(\C) : \M \mapsto \C\bt_\D \M. 
\end{equation}
We will call this homomorphism the {\em Picard induction}.   A different kind of induction 
(of central autoequivalences) will be considered in Section~\ref{section 3}.

The kernel of $P_\D^\C$  was described in \cite[Proposition 3.11]{BN}.
Namely, let 
\[
\C =\bigoplus_{\alpha \in \Sigma}\, \C_\alpha
\]
be the decomposition of $\C$ into a direct sum of $\D$-module subcategories. Then 
\begin{equation}
\label{Ker PDC}
\Ker(P_\D^\C) = \{ \C_\alpha ,\, \alpha \in \Sigma \mid \text{$\C_\alpha$ is an invertible $\D$-module category} \}.
\end{equation}


\section{Induction of central autoequivalences}
\label{section 3}

\subsection{Induction $ \Aut(\A) \to \Aut^{br}(\Z(\A))$}

Let $\A$ be a fusion category. Recall that the objects of the center $\Z(\A)$ of $\A$ 
are pairs $(Z,\, \gamma)$, where $Z$ is an object of $\A$ and $\gamma= \{\gamma_X : X \ot Z \to Z\ot X\}$
is a natural isomorphism satisfying certain coherence axioms, see, e.g., \cite[7.13]{EGNO}. 
The fusion category $\Z(\A)$ has a canonical braiding and 
there is an   induction homomorphism
\begin{equation}
\label{induction from Aut}
\ind: \Aut(\A) \to \Aut^{br}(\Z(\A)) : \alpha \mapsto \ind(\alpha),
\end{equation}
where $\ind(\alpha)(Z,\, \gamma) = (\alpha(Z),\, \gamma^\alpha)$ and   $\gamma^\alpha$
is defined by the following commutative diagram
\begin{equation} 
 \xymatrix{
 X\otimes \alpha(Z)\ar[rr]^{\gamma^\alpha_X}\ar[d] & &  \alpha(Z)\otimes X\ar[d] \\
 \alpha(\alpha^{-1}(X)) \otimes \alpha(Z)\ar[d]_{J_{\alpha^{-1}(X),Z}} & & \alpha(Z) \otimes \alpha(\alpha^{-1}(X))\ar[d]^{J_{Z,\alpha^{-1}(X)}}\\
 \alpha(\alpha^{-1}(X)\otimes Z)\ar[rr]^{\alpha(\gamma_{\alpha^{-1}(X)})} & & \alpha(Z\otimes \alpha^{-1}(X)).
 }
\end{equation}
Here $\alpha^{-1}$ is a quasi-inverse of $\alpha$ and $J_{X,Y}: \alpha(X)\ot \alpha(Z) \xrightarrow{\sim} \alpha(X\ot Z)$
is the tensor functor structure of $\alpha$. 

\subsection{The kernel of induction}

Let us recall several canonical  homomorphisms associated to a fusion category $\A$. 

Let $U(\A)$ denote the universal grading group of $\A$ and  let $U(\A)_{ab}$ denote the maximal abelian quotient
of $U(\A)$. The objects $Z$ in $\Z(\A)$ such that $Z =\be$ as an object of $\A$  form an abelian group isomorphic
to $\widehat{U(\A)}$ (i.e., the group of linear characters of $U(\A)$), see \cite{GN}.  
Indeed, central object structures on $\be$ are in bijection with tensor 
automorphisms of  the identity endofunctor of $\A$. This yields an injective homomorphism 
\begin{equation}
\label{UC to Inv}
\widehat{U(\A)} \to\Inv(\Z(\A)).
\end{equation}

Next, the forgetful functor $F:\Z(\A)\to \A$ restricts to a homomorphism between groups of invertible objects:
\begin{equation}
\label{forgetful inv}
F: \Inv(\Z(\A))  \to  \Inv(\A).	
\end{equation}

For any invertible object $X\in \A$ the conjugation by $X$ is a tensor autoequivalence of $\C$, thus there is  
a group homomorphism 
\begin{equation}
\label{conj}
\conj:
\Inv(\A)\to \Aut(\A).
\end{equation}

\begin{proposition} 
\label{sse for ker ind}
The  sequence of group homomorphisms
\begin{equation}
\label{exact sequence for ind}
\xymatrix{
	0 \ar[r] & \widehat{U(\A)} \ar[r] & \Inv(\Z(\A)) \ar[r]^{F}
	& \Inv(\A) \ar[r]^{\conj} & \Aut(\A) \ar[r]^{\ind\quad}
	& \Aut^{br}\left(\Z(\A)\right)
}
\end{equation}
is exact.
\end{proposition}
\begin{proof}
The exactness at $\Inv(\Z(\A))$ is obvious.  To see that  the sequence is exact at  $\Inv(\A)$ observe
that a natural tensor isomorphism between the conjugation functor 
\[
V \mapsto X\ot V \ot X^*
\]
and $\id_\A$ is the same thing as a central
structure on $X$. It remains to  establish exactness at  $\Aut(\A)$. 
For $\alpha\in \Aut(\A)$ let $\A_\alpha$ denote  the invertible $\A$-bimodule category 
corresponding to the induced autoequivalence $\ind(\alpha) \in \Aut^{br}\left(\Z(\A)\right)$ 
under isomorphism \eqref{Phi}. This category  is equivalent to the 
regular category $\A$ as a right $\A$-module category and left action of $\A$ on $\A_\alpha$ is given by
\begin{equation}
\label{C alpha}
(X,\, V) \mapsto \alpha(X) \ot V,
\end{equation}
for all $X\in \A$ and $V\in \A_\alpha$, 
see \cite[Example 6.4]{NR}.  Braided autoequivalence $\ind(\alpha)$ is trivial if and only if there is an $\A$-bimodule
equivalence between $\A_\alpha$ and $\A$. Such an equivalence is given by $V\mapsto X\ot V$ for an invertible
object $X$ in $\A$ such that $\alpha$ is the conjugation by $X$. Thus, the result follows from isomorphism~\eqref{Phi}.
\end{proof}

\begin{remark}
The exact sequence \eqref{exact sequence for ind} can also be obtained by combining two exact sequences
from \cite[Theorem 1.3]{GP}. 
\end{remark}

\subsection{The image of induction}
\label{section 3.3}

Let $\C$ be a non-degenerate braided fusion category and let $A$ be a {\em Lagrangian algebra} in $\C$.
By definition \cite{DMNO}, this means that $A$ is a connected \'etale 
(i.e., commutative and separable) algebra in $\C$ such that $\FPdim(A)^2=\FPdim(\C)$.
Let $\C_A$ be the category of left $A$-modules in $\C$. It is a fusion category with tensor product $\ot_A$.

There is a canonical braided tensor equivalence 
\begin{equation}
\label{iotaA}
\iota_A : \C \xrightarrow{\sim}  \Z(\C_A) : Z \mapsto A\ot Z.
\end{equation}
Let 
\begin{equation}
\label{ind}
\ind_A : \Aut(\C_A) \to \Aut^{br}(\Z(\C_A))
\end{equation} 
denote the induction homomorphism.

\begin{theorem}
\label{alphaA=A}
Let $\alpha$ be a braided tensor autoequivalence of $\C$.  The following condition are equivalent:
\begin{enumerate}
\item[(i)] there is an algebra isomorphism $\alpha(A) \cong A$,
\item[(ii)] there is $\gamma\in \Aut(\C_A)$ such that  $\alpha = \iota_A^{-1} \circ \ind_A(\gamma) \circ \iota_A$. 
\end{enumerate}
\end{theorem}
\begin{proof}
Suppose that there is an algebra isomorphism $\phi: \alpha(A) \xrightarrow{\sim} A$. 
Define an autoequivalence $\gamma \in \Aut(\C_A)$ as follows. Given an $A$-module $X$ in $\C$
with the action $p:A\ot X \to X$  we set $\gamma(X) =\alpha(X)$ as an object in $\C$, with the action
\[
A \ot \gamma(X) =  A \ot \alpha(X) \xrightarrow{\phi^{-1}\ot \id_{\alpha(X)}} 
\alpha(A)\ot \alpha(X) \cong \alpha(A\ot X) \xrightarrow{\alpha(p)} \alpha(X).
\]
Then $\iota_A  \alpha  \iota_A^{-1} (A \ot Z)  \cong A \ot \alpha(Z) = \gamma(A\ot Z)$ as $A$-modules
and its central structure is determined by that of $Z $. This means that 
$\iota_A \circ \alpha\circ \iota_A^{-1} =\ind_A(\gamma)$.

Conversely, suppose that $\alpha = \iota_A^{-1} \circ \ind_A(\gamma) \circ \iota_A$ for some $\gamma\in \Aut(\C_A)$.
Let $F_A: \Z(\C_A)\to \C_A$ be the forgetful functor and let $I_A:\C_A\to \Z(\C_A)$ be its right adjoint.
Note that $I_A(\be)$ is a Lagrangian algebra in $ \Z(\C_A)$ and $A\cong \iota_A^{-1}(I_A(\be))$. 
For any $\gamma\in  \Aut(\C_A)$ we have a natural tensor isomorphism 
\begin{equation}
\label{tensoriality}
F_A \circ \ind(\gamma) \cong  \gamma \circ F_A.
\end{equation}
Taking adjoints of both sides and replacing $\gamma$ by its inverse we obtain 
natural isomorphism $\ind(\gamma) \circ I_A \cong I_A \circ \gamma$ satisfying a  multiplicative property
corresponding to \eqref{tensoriality} being an isomorphism of  tensor functors.  
Applying both sides of the last isomorphism to $\be$ we obtain  an algebra isomorphism
\[
\ind(\gamma) \left( I_A(\be) \right) \cong  I_A(\be),
\]
which is equivalent to $\alpha(A)\cong A$. 
\end{proof}

\begin{remark}
It follows from Theorem~\ref{alphaA=A} that 
the image of the induction homomorphism \eqref{ind} is
the stabilizer of (the isomorphism class of) 
a Lagrangian algebra $A\in\C$ in $\Aut^{br}(\C)$.
\end{remark}

\subsection{Orbits of the action of the Brauer-Picard group on the categorical Lagrangian Grassmannian}
\label{section 3.4}

Let $\A$ be a fusion category. Denote by $\L(\A)$ the groupoid  of Lagrangian algebras in $\Z(\A)$.
It is known that $\L(\A)$  is equivalent to the groupoid of indecomposable $\A$-module categories via $A\mapsto \Z(\A)_A$.
The group $\Aut^{br}(\Z(\A))$ (i.e., the Brauer-Picard group of $\A$) acts on $\L(\A)$ in an obvious way. 

\begin{remark}
The groupoid $\L(\A)$ can be thought of as a categorical analogue of a Lagrangian Grassmannian. 
In \cite{NR} the action of $\Aut^{br}(\Z(\A))$  on the subset of $\L(\A)$  consisting of  algebras coming 
from Lagrangian subcategories of $\Z(\A)$ was considered.
 \end{remark}
 
 The following result extends Theorem~\ref{alphaA=A} from the previous section.
 
 \begin{proposition}
 \label{orbits}
 Let $A$ and $B$ be Lagrangian algebras in $\Z(\A)$.  There exists a tensor equivalence  $\Z(\A)_A \cong \Z(\A)_B$ if and only if
 there exists  $\alpha\in \Aut^{br}(\Z(\A))$ such that $\alpha(A) \cong B$ as algebras in $\Z(\A)$. 
 \end{proposition}
\begin{proof}
Suppose there is a tensor equivalence $\phi : \Z(\A)_A \to \Z(\A)_B$.  Let 
$\ind(\phi): \Z(\Z(\A)_A)\to \Z(\Z(\A)_B)$ be the induced braided tensor equivalence and let 
\[
\iota_A: \Z(\A)\to  \Z(\Z(\A)_A),\qquad \iota_B: \Z(\B)\to \Z(\Z(\A)_B)
\]
be braided equivalences defined in the previous section. Then  $\alpha= \iota_B^{-1} \circ \ind(\phi) \circ \iota_A$
is a braided autoequivalence of $\Z(\A)$ that maps $A$ to $B$ (the verification of this fact
is completely parallel to one in the proof of Theorem~\ref{alphaA=A}).

Conversely, if $\alpha\in  \Aut^{br}(\Z(\A))$ then there is a tensor equivalence 
\[
\Z(\A)_A \cong \Z(\A)_{\alpha(A)}
\]
given by $X\mapsto \alpha(X)$ for all objects $X\in \Z(\A)_A $. 
\end{proof}

\begin{remark}
\label{transitivity on L}
Proposition~\ref{orbits} implies  that orbits of the action of $\Aut^{br}(\Z(\A))$ on the set of isomorphism classes 
of Lagrangian algebras in $\Z(\A)$ are parameterized by equivalence classes of fusion categories
categorically Morita equivalent to $\A$. 

Equivalently,  we have a parameterization of orbits of  the action of $\BrPic(\A)$ on the set of equivalence classes
of indecomposable left $\A$-module categories. Namely, two $\A$-module categories belong to the same orbit
if and only if the corresponding dual fusion categories are equivalent.  A related result is established 
by Galindo and Plavnik in  \cite[Theorem 1.4]{GP}.
\end{remark}


\section{Representation categories of twisted group doubles}
\label{section 4}

Let $G$ be a finite group and let $\omega \in H^3(G,\, k^\times)$.  

\subsection{Invertible objects of  the center of a pointed fusion category}
\label{invertibles}

Let $Z(G)$ denote the center of $G$.  For  any $a\in Z(G)$ let
\begin{equation}
\label{betaa}
\beta_a  (x,\,y)= \frac{\omega(a,\,x,\, y) \omega(x,\,y,\,a)}{\omega(x,\,a,\,y)},\qquad x,y\in G. 
\end{equation}
It is known that $\beta_a$ is a $2$-cocycle and that the map 
\begin{equation}
\label{beta map}
\beta: Z(G) \to H^2(G,\,k^\times) :  a \mapsto \beta_a
\end{equation}
is a group homomorphism. 


The invertible objects of $\Z(\C(G,\,\omega))$ are well known, see, e.g., \cite{DPR}.
The exact sequence \eqref{ZC inv sequence} in the next proposition 
can be found in \cite[Example 6.2]{GP}. We include its proof for the sake of completeness. 

For any group $G$ let $\widehat{G}$ denote the group of linear characters of $G$.  

\begin{proposition}
\label{sequence for invertibles}
The following sequence
\begin{equation}
\label{ZC inv sequence}
0 \to \widehat{G} \to \Inv(\Z(\C(G,\,\omega)) \xrightarrow{F} Z(G) \xrightarrow{\beta} H^2(G,\,k^\times)
\end{equation}
is exact. Here $F$ is given by the forgetful functor. 
\end{proposition}
\begin{proof}
The central structures on the identity object of $\C(G,\,\omega)$ are parameterized by
linear characters of $G$. It is clear that in order to have a central structure
the invertible object of $\C(G,\,\omega)$  must correspond to an element of $Z(G)$.  Finally, it follows
from \eqref{betaa} that $a\in Z(G)$ admits a central structure if and only if $\beta_a$ is a coboundary. 
\end{proof}

\begin{remark}
Proposition~\ref{sequence for invertibles} appears (in a different form) in \cite[Proposition 5.2]{MN2}.
It can also be derived from the exact sequence \eqref{exact sequence for ind}.
\end{remark}

\begin{corollary}
\label{ZC  pointed}
The category $\Z(\C(G,\,\omega))$ is pointed if and only if $G$ is abelian and $\beta$ is zero. 
\end{corollary}

\begin{corollary}
\label{ZC  antipointed}
Suppose that $G$ is abelian.  Then $\Z(\C(G,\,\omega))_{pt}$ is Lagrangian if and only if $\beta$ is injective. 
\end{corollary}
\begin{proof}
Note that $\Z(\C(G,\,\omega))_{pt}$ contains the  Lagrangian subcategory $\L=\Rep(G)$ 
consisting of central objects supported on $\be$. 
Clearly, $\Z(\C(G,\,\omega))_{pt} =\L$ if and only if the forgetful homomorphism
$F: \Inv(\Z(\C))\to \Inv(\C)$ is trivial, i.e., if and only if $\beta$ is injective. 
\end{proof}

\subsection{Image of induction in the pointed case}
\label{ptd case}
Let $G$ be a finite group and let $\omega\in H^3(G,\,k^\times)$. 

Let $\OutStab(\omega)= \Stab(\omega)/\Inn(G)$ denote the subgroup of $\Out(G)$ consisting 
of classes of automorphisms $a$ such that $\omega\circ (a\times a\times a)$ is cohomologous to $\omega$. 

Let $I(G,\,\omega)$ be the image of induction $\Aut(\C(G,\,\omega))\to \Aut^{br}(\Z(\C(G,\,\omega)))$. 

\begin{proposition}
\label{ptd image of induction}
There is an exact sequence
\begin{equation}
\label{Sequence for IG}
Z(G) \xrightarrow{\beta}  H^2(G,\, k^\times) \xrightarrow{\ind} I(G,\,\omega) \to  \OutStab(\omega) \to 0
\end{equation}
\end{proposition}
\begin{proof}
Using the exact sequence \eqref{exact sequence for ind} we see that 
\[
I(G,\,\omega) \cong \Aut(\C(G,\, \omega))/\mbox{Ker}(\ind) \cong \Aut(\C(G,\, \omega))/\mbox{Im}(\conj),
\]
where $\conj$ is given by \eqref{conj}. The image of $\conj$ in $\Aut(\C(G,\, \omega))$ is 
generated by $\Inn(G) =G/Z(G)$ and the image of $\beta:Z(G)\to H^2(G,k^\times)\subset \Aut(\C(G,\, \omega))$ 
defined in \eqref{beta map} (note that  the conjugation by $a\in Z(G)$ gives rise to a non-trivial tensor autoequivalence 
of $\C(G,\, \omega)$ precisely when $\beta_a$ is non-trivial in $H^2(G,k^\times)$).  By \eqref{AutCGw},
this implies the statement.
\end{proof}

\begin{corollary}
\label{combining2seq}
We have an exact sequence
\begin{equation}
0 \to \widehat{G} \to \Inv(\Z(\C(G,\,\omega)) \xrightarrow{F} Z(G) \xrightarrow{\beta} H^2(G,\,k^\times)
\xrightarrow{\ind} I(G,\,\omega) \to  \OutStab(\omega) \to 0.
\end{equation}
\end{corollary}

\begin{corollary}
\label{Indsurj}
Let $\C$ be a non-degenerate braided fusion category and let $\L =\Rep(G)$
be a Lagrangian subcategory of $\C$ such that $\alpha(\L) \cong \L$ for every  $\alpha\in \Aut^{br}(\C)$.
Then $\C \cong \Z(\C(G,\, \omega))$ for some $\omega\in H^3(G,\, k^\times)$ and the induction homomorphism
\begin{equation}
\label{indsurj}
\Aut(\C(G,\, \omega)) \to   \Aut^{br}(\Z(\C(G,\, \omega)) 
\end{equation}
is surjective, i.e., $\Aut^{br}(\Z(\C(G,\, \omega))  =I(G,\,\omega)$.
\end{corollary}
\begin{proof}
This follows from Theorem~\ref{alphaA=A} and Proposition~\ref{ptd image of induction}.
\end{proof}

\subsection{Braided equivalences between representation categories of twisted group doubles}
\label{sect Deepak}

Let $G$ be a finite group and let $A \subset G$ be a normal abelian subgroup.  Let $K:=G/A$ be the quotient group,
so that there is an extension
\begin{equation}
0 \to A \to G \to K \to 0.
\end{equation}
Such an extension is determined up to an isomorphism by the action of $K$ on $G$ 
(denoted by $(x,\,a) \mapsto x\cdot a$ for $x\in K,\, a\in A$)
and the cohomology class of a $2$-cocycle
$\kappa \in Z^2(K,\, A)$, so that  elements of $G$ are identified with pairs $(a,\,x)\in A\times K$ and the
multiplication is given by
\[
(a,\,x)\, (b,\,y) = (a (x\cdot b) \kappa(x,\,y) ,\, xy),\qquad a,b\in A,\, x,y\in K.  
\]
It was shown in \cite{MN1, N} that the fusion category $\Vec_G$ is categorically Morita equivalent to 
$\Vec_{\widehat{A}\rtimes K}^\omega$, where  $\widehat{A}$ is the dual of the $K$-module $A$
and   the $3$-cocycle $\omega\in Z^3(\widehat{A}\rtimes K,\,k^\times)$
is defined by
\begin{equation}
\label{kappa rho}
\omega((\rho_1,x_1),\, (\rho_2,x_2),\, (\rho_3,x_3)) =  \rho_1(\kappa(x_2,\, x_3)),
\end{equation}
for all $\rho_1,\rho_2,\rho_3 \in \widehat{A}$ and $x_1,x_2,x_3\in K$. 

\begin{remark}
In fact, in \cite{MN1, N, U} {\em all} pairs of Morita equivalent pointed fusion categories 
(i.e., all pairs of of twisted group doubles with braided tensor equivalent representation categories) 
were classified.  In this paper we will only use the special case described above. 
\end{remark}

Thus, there exists a braided equivalence 
\begin{equation}
\label{G and AK}
\Z(\Vec_G) \cong \Z(\Vec_{\widehat{A}\rtimes K}^\omega).
\end{equation}
So for computational  purposes  the double of $G$ can be replaced by the {\em twisted} double 
of an easier group $\widehat{A}\rtimes K$.


\section{Twisted doubles of elementary abelian $p$-groups and their Brauer-Picard groups}
\label{section 5}

Let $p$ be a prime, let $n$ be a positive integer, and let $V_n$ denote the elementary abelian $p$-group of order $p^n$. 
We have $\Aut(V_n) = GL_n(\mathbb{F}_p)$.
Below we will also view $V_n$ as an $n$-dimensional vector space over the field $\mathbb{F}_p$ with $p$ elements. 
We will denote  $V_n^* = \Hom(V_n,\, \mathbb{F}_p)$  the dual vector space.

As before, for $\omega\in H^3(V_n,\, k^\times)$ we denote $\C(V_n,\, \omega)$ the category of $V_n$-graded vector spaces
with the associativity constraint twisted by $\omega$. 

For a vector space $W$ we denote 
\[
\bigwedge(W)= \bigoplus_{i=0}^\infty\, {\bigwedge}^i(W)\quad \text{and}\quad 
\Sym(W)= \bigoplus_{i=0}^\infty\, \Sym^i(W)
\]
the alternating and symmetric algebras of $W$. When $W$ has a basis $\{w_1,\dots, w_n\}$ we also write
$\bigwedge(W)=\bigwedge(w_1,\dots, w_n)$ and $\Sym(W)=\Sym(w_1,\dots, w_n)$.
\subsection{The Brauer-Picard group of $\C(V_n,\, \omega)$ when $p$ is odd}
\label{BP p odd section}
Let $p$ be an odd prime.


The cohomology ring $H^\bullet(V_n,\, \mathbb{F}_p)$ is well known (see, e.g., \cite{A}), namely
\[
H^\bullet(V_n,\, \mathbb{F}_p) = \bigwedge(x_1,\dots, x_n) \ot_{\mathbb{F}_p} \mathbb{F}_p[y_1,\dots, y_n],
\]
where $\deg(x_i)=1$ and $\deg(y_i)=2$ for all $i=1,\dots,n$. 

The cocycles representing generators $x_i$ and $y_i$ can be explicitly described as follows 
(below we identify cocycles with the cohomology classes they represent).
Define $x_i: V_n\to \mathbb{F}_p$ and $y_i: V_n \times V_n \to \mathbb{F}_p$ by
\begin{equation}
\label{x i}
x_i(v) = v_i
\end{equation}
and 
\begin{equation}
\label{y i}
y_i(u,\,v) = \begin{cases} 0 & \text{ if } u_i+v_i < p, \\ 1 & \text{ if } u_i+v_i\geq p, \end{cases}
\end{equation}
for all $v=(v_1,\dots v_n),\, u=(u_1,\dots u_n)$ in $V_n$ and $i=1,\dots n$.  
Here we view $u_i,\,v_i$ as elements of $\{0,1,\dots, p-1\}$ and add them as usual integers. 

In particular,
\begin{equation}
\label{H3 VnFp}
H^{3}(V_n,\,\mathbb{F}_p)  = {\bigwedge}^3(x_1,\dots, x_n) \,\bigoplus\, \mathbb{F}_p \langle x_i\cup y_j \mid i,j=1,\dots,n \rangle,
\end{equation}
where $\cup$ denotes the cup product. Note that the second summand in \eqref{H3 VnFp} is isomorphic to $V_n^*\ot V_n^*$
as a $GL_n(\mathbb{F}_p)$-module. 

\begin{proposition}
\label{H3 when p odd}
There is an isomorphism of $GL_n(\mathbb{F}_p)$-modules:
\begin{equation}
\label{H3Ep=odd}
H^3(V_n ,\, k^\times) \cong {\bigwedge}^3 (V_n^*) \,\bigoplus\, \Sym^2(V_n^*).
\end{equation}
\end{proposition}
\begin{proof}
The automorphism of $k^\times$ given by $\xi\mapsto \xi^p$ 
yields an exact sequence of abelian groups:
\[
\xymatrix{
	0 \ar[rr] && \mathbb{F}_p \ar[rr] && k^\times \ar[rr] && k^\times \ar[rr] && 0,
}
\]
where $\mathbb{F}_p$ is identified with the group of $p$th roots of $1$ in $k$.
This, by functoriality,  yields a long exact sequence of $GL_n(\mathbb{F}_p)$-modules:
\[\xymatrix@R-2pc{
	\cdots \ar[r]  & H^{m-1}(V_n,\,\mathbb{F}_p) \ar[r] & H^{m-1}(V_n,\,k^\times) \ar[r]^{0} & H^{m-1}(V_n,\,k^\times)&
	\\
	 \ar[r]& H^{m}(V_n,\,\mathbb{F}_p) \ar[r] & H^{m}(V_n,\,k^\times) \ar[r]^{0} & H^{m}(V_n,\,k^\times)\ar[r]&\cdots
}\]
Note that the map $H^{m-1}(V_n,\,k^\times) \to H^{m-1}(V_n,\,k^\times)$ induced by taking the $p$th
power is zero since the exponent of $H^m(V_n,\,k^\times)$ is $p$.  The latter fact follows from  isomorphism
$H^m(V_n,\,k^\times) = H^{m+1}(V_n,\,\mathbb{Z})$ and  the K\"unneth formula for the direct product in cohomology.

In particular, there is a short exact sequence of $GL_n(\mathbb{F}_p)$-modules
\begin{equation}
\label{seqH2H3H3}
0\to H^{2}(V_n,\,k^\times) \xrightarrow{\delta} H^{3}(V_n,\,\mathbb{F}_p) \to H^{3}(V_n,\,k^\times)\to 0.
\end{equation}
We claim that the image of inclusion 
\begin{equation}
\label{H2toH3odd}
\delta: H^{2}(V_n,\,k^\times) = {\bigwedge}^2(V_n^*) \to H^{3}(V_n,\,\mathbb{F}_p) 
\end{equation}
is the subspace of  $\mathbb{F}_p\langle x_i\cup y_j \mid i,j=1,\dots,n\rangle \subset H^{3}(V_n,\,\mathbb{F}_p)$ 
consisting  of  cohomology classes of the form 
$\sum_{i,j=1}^n\, a_{ij}x_i\cup y_j $, where $A=\{a_{ij}\}$ is an anti-symmetric matrix over $\mathbb{F}_p$. 

Indeed, \eqref{H2toH3odd} is given by the connecting homomorphism which is explicitly computed as follows.
Fix a primitive $p$-th root of unity $\xi$ in $k$ and consider homomorphism
\[
\mathbb{F}_p \to k^\times: a\mapsto \xi^a.
\]
A class in $H^{2}(V_n,\,k^\times)$ is represented by a $2$-cocycle $\mu :V_n\times V_n \to k^\times$ given by
\[
\mu(u,\,v) =\xi^{(Au,\,v)},\qquad u,\,v\in V_n,
\]
where $A=\{a_{ij}\}$ is an $n$-by-$n$ matrix over $\mathbb{F}_p$ and $(Au,\,v) =\sum_{i,j}^n\, a_{ij}u_iv_j$
is viewed as a non-negative integer.

Let $\lambda$ be a fixed $p$th root of $\xi$ in $k$.  Take a $2$-cochain $\nu: V_n\times V_n \to k^\times$ defined by
\[
\nu(u,\,v) =\lambda^{(Au,\,v)},\qquad u,\,v\in V_n.
\]
For non-negative integers $a,\,b$  let us denote $\lbrace a,b\rbrace$ the integral part of $\frac{a+b}{p}$. 
We have 
\begin{eqnarray*}
\nu({u},{v})\nu({u'},{v}) &=& \nu({u}+{u'},{v})\,\xi^{\lbrace(A{u},{v}),(A{u'},{v})\rbrace}, \\
\nu({u},{v})\nu({u},{v'}) &=&\nu({u},{v}+{v'})\,\xi^{\lbrace(A{u},{v}),(A{u},{v'})\rbrace},
\end{eqnarray*}
for all $u,\,u',\,v,\,v'\in V_n$. Using these identities we compute the differential of $\nu$:
\[
	d\nu({u},{v},{w})
	=\frac{\nu({u}+{v},{w})\nu({u},{v})}
	{\nu({u},{v}+{w})\nu({v},{w})}
	= \frac{\xi^{\lbrace (A{u},{v}),(A{u},{w}) \rbrace}}
	{\xi^{\lbrace (A{u},{w}),(A{v},{w}) \rbrace}}, \qquad u,\,v,\,w\in V_n.
\]
Fix $k,\,l\in\{1,2,\dots,n\}$ and take $A$ such that  $a_{ij}=1$ if $i=k$ and $j=l$ and $a_{ij}=0$ otherwise.
Then the previous calculation yields
\[
d\nu({u},{v},{w}) =\frac{\xi^{u_k\lbrace v_l,w_l\rbrace}} {\xi^{w_l\lbrace u_k,v_k\rbrace}}.
\]
Since $d\nu(u,\,v,\,w) =\xi^{\delta(\mu)(u,v,w)}$, we conclude that in this case
\[
\delta(\mu) (u,\,v,\,w)  = u_k\lbrace v_l,w_l\rbrace - w_l\lbrace u_k,v_k\rbrace. 
\]
Comparing this with \eqref{x i} and \eqref{y i} (note that $y_i(u,\,v) =\lbrace u_i,\, v_i\rbrace$)
we conclude that the image of $\delta$ is  spanned by 
\begin{equation}
\label{span of image odd}
x_k\cup y_l - x_l \cup y_k,\, k,l=1,\dots,n,
\end{equation}
as claimed.  The quotient of $V_n^*\ot V_n^*$ by this space is isomorphic to $\Sym^2(V_n^*)$ via the symmetrization map. 
Since
\[
H^3(V_n, \, k^\times) = H^3(V_n,\, \mathbb{F}_p) / \text{Image}(\delta: H^2(V_n, \, k^\times) \to H^3(V_n,\, \mathbb{F}_p)),
\]
the statement follows from the last claim and \eqref{H3 VnFp}. 
\end{proof}

Given a cohomology  class $\omega \in H^3(V_n ,\, k^\times)$  denote
\begin{equation}
\label{alt+sym}
\omega = \omega_{alt} +\omega_{sym}, \qquad  \omega_{alt}  \in {\bigwedge}^3 (V_n^*) ,\quad  \omega_{sym} \in \Sym^2(V_n^*) 
\end{equation}
the decomposition of $\omega$ from \eqref{H3Ep=odd}.

Let $V$ be a vector space. Consider the interior derivation
\[
\iota: V \ot {\bigwedge}^3 (V^*) \to {\bigwedge}^2 (V^*) : v \ot \phi \mapsto \iota_v(\phi),
\]
given by  
\[
\iota_v(x\wedge y \wedge z) = \langle v,\, z\rangle x\wedge y -   \langle v,\, y\rangle x\wedge z + \langle v,\, x\rangle y\wedge z,
\]
for all $v\in V,\, x,y,z\in V^*$, and extended to ${\bigwedge}^3 (V^*)$ by linearity. 

The {\em radical} of $\phi\in {\bigwedge}^3 (V^*)$ is defined as
\begin{equation}
\label{radical}
\Rad(\phi) := \{ u \in V \mid \iota_u(\phi) =0\}.
\end{equation}
We say that $\phi \in \bigwedge^3 (V^*)$ is {\em non-degenerate} if $\Rad(\phi)=0$.

\begin{proposition}
\label{vinimageFodd}
Let $v\in V_n$ and let $\omega\in H^3(V_n,\, k^\times)$. 
Then $v$, regarded as a simple object in $\C(V_n,\, \omega)$,  is
in the image of  $\Inv(\Z(\C(V_n,\, \omega))) \to  \Inv(\C(V_n,\, \omega))$ if and only if $v \in \Rad(\omega_{alt})$.
\end{proposition}
\begin{proof}
We claim that in this case the homomorphism 
\[
\beta: V_n \to H^2(V_n,\,k^\times) ={\bigwedge}^2(V_n^*)
\] 
defined by \eqref{betaa} and \eqref{beta map} is given by
\begin{equation}
\label{betavgeneral}
\beta(v)= \iota_v(\omega_{alt}),\quad v\in V_n.
\end{equation}
Let us first prove this claim when $\omega = \omega_{alt} \in {\bigwedge}^3(V_n^*)$ (i.e., when the symmetric part of $\omega$
in \eqref{alt+sym} is trivial).  Such an $\omega$ is a linear combination of $3$-cocycles $\omega_{ijk}$ ($i,j,k$ are distinct
elements of $\{1,\dots,n\}$), given by $\omega(u,\,v,\,w) =\xi^{u_iv_jw_k}$, where $\xi$ is a fixed primitive $p$th root of $1$ in $k$. 
So we may assume that $\omega=\omega_{ijk}$. We have
\[
\beta(v) = v_i x_j \wedge x_k - v_j x_i \wedge x_k + v_k x_i\wedge x_j, \qquad v =(v_1,\dots, v_n)\in V_n,
\]
where $x_i$ are defined by \eqref{x i}. Thus, $\beta(v) = \iota_v(\omega_{ijk})$ and \eqref{betavgeneral}
is true in this case.

Next, let us prove the claim when $\omega = \omega_{sym}\in \Sym^2(V_n^*)$. We need to check that
$\beta=0\in H^2(V_n,\, k^\times)$ in this case.  We may assume that $\omega$ is the image of
$x_i\cup y_j \in H^3(V_n,\mathbb{F}_p)$ under \eqref{seqH2H3H3} i.e.,
\[
\omega(u,\,v,\,w)= \xi^{x_i(u) y_j(v,w)},\qquad u,v,w\in V_n.
\]
Since $y_j$ is symmetric (see \eqref{y i}) we conclude that 
 $\beta(v)$ is the image of $v_iy_j$ under the homomorphism $H^2(V_n,\, \mathbb{F}_p) \to  H^2(V_n,\, k^\times)$
and, hence, is trivial. 

So, \eqref{betavgeneral} is true for all $\omega\in H^3(V_n,\, k^\times)$ and the result follows from Proposition~\ref{sequence for invertibles}.
\end{proof}

\begin{corollary}
\label{ptd iff}
The category $\Z(\C(V_n,\, \omega))$ is pointed if and only if $\omega_{alt} =0$. 
\end{corollary}

\begin{corollary}
\label{walt nondeg}
$\Z(\C(V_n,\, \omega))_{pt} =\Rep(V_n)$ if and only if $\omega_{alt}$ is non-degenerate.
In this case $\Z(\C(V_n,\, \omega))_{pt}$ is the trivial component of the universal grading of $\Z(\C(V_n,\, \omega))$
and the universal grading group of $\Z(\C(V_n,\, \omega))$ is $\Hom(V_n,\, k^\times)$, the dual group of $V_n$.
\end{corollary}
\begin{proof}
This follows from Corollary~\ref{ZC  antipointed}.
\end{proof}

\begin{theorem}
\label{exact seq for BrPic in nondeg case}
Let  $\omega\in H^3(V_n,\, k^\times)$ be such that $\omega_{alt}$ is non-degenerate. There is
an exact sequence of groups:
\begin{equation}
\label{exact sequence for Vn}
0\to V_n \xrightarrow{\iota(\omega_{alt})} {\bigwedge}^2(V_n^*) \to \BrPic(\C(V_n,\, \omega)) \to \Stab_{V_n}(\omega) \to 0.
\end{equation}
\end{theorem}
\begin{proof}
This follows from Proposition~\ref{ptd image of induction} and Corollary~\ref{Indsurj} applied to $G=V_n$.
\end{proof}

Theorem~\ref{exact seq for BrPic in nondeg case} implies that the Brauer-Picard group of $\C(V_n,\, \omega)$ is an extension
 of $\Stab_{V_n}(\omega) = \Stab(\omega_{alt}) \cap \Stab(\omega_{sym})$ by an elementary abelian $p$-group. 

\subsection{The Brauer-Picard group of $\C(V_n,\, \omega)$ when $p=2$}

It is known (see, e.g., \cite{A}) that  
\[
H^\bullet(V_n,\, \mathbb{F}_2) =  \mathbb{F}_2[x_1,\dots, x_n],
\]
where $x_i,\, i=1,\dots,n$ are one-dimensional generators represented by $1$-cocycles 
\begin{equation}
\label{x i when p=2}
x_i(v) = v_i, \qquad \text{where } v =(v_1,\dots v_n)\in V_n.
\end{equation}

\begin{proposition}
\label{H3 when p=2}
There is an isomorphism of $GL_n(\mathbb{F}_2)$-modules:
\begin{equation}
\label{H3Ep=2}
H^3(V_n ,\, k^\times) \cong \Sym^3 (x_1,\dots, x_n) / \mathbb{F}_2 \langle x_i^2x_j+x_ix_j^2 \mid \, i,j=1,\dots,n\rangle.
\end{equation}
\end{proposition}
\begin{proof}
The argument is similar to that of Proposition~\ref{H3 when p odd}.
There is a short exact sequence of $GL_n(\mathbb{F}_2)$-modules
\[
0\to H^{2}(V_n,\,k^\times) \to H^{3}(V_n,\,\mathbb{F}_2) \to H^{3}(V_n,\,k^\times)\to 0.
\]
The same computation as in the proof of Proposition~\ref{H3 when p odd} shows
image of inclusion of $H^{2}(V_n,\,k^\times) = \bigwedge^2(V_n^*)$ into 
$H^{3}(V_n,\,\mathbb{F}_p)  =  \Sym^3(x_1,\dots, x_n)$ is spanned by elements 
$x_k \cup y_l + x_l \cup y_k$ with $k,l=1,\dots,n$ (see \eqref{span of image odd}). Since  for $p=2$ we have $y_k=x_k^2$,    
the result follows. 
\end{proof}


\begin{proposition}
\label{H3 when p even}
There is a short exact sequence of $GL_n(\mathbb{F}_2)$-modules:
\begin{equation}
\label{H3Ep=even}
0\to \Sym^2(V_n^*) \to H^3(V_n ,\, k^\times)  \xrightarrow{\pi} {\bigwedge}^3 (V_n^*) \to 0.
\end{equation}
\end{proposition}
\begin{proof}
Using Proposition~\ref{H3 when p=2} we see that $H^3(V_n ,\, k^\times)$ contains
a $GL_n(\mathbb{F}_2)$-submodule  spanned by $x_i^2x_j,\, i,j=1,\dots,n$ (modulo $\mathbb{F}_2 \langle x_i^2x_j+x_ix_j^2 \mid 
 i,j=1,\dots,n\rangle$). Clearly, this submodule is isomorphic to  $\Sym^2(V_n^*)$ and the corresponding quotient
is isomorphic to  ${\bigwedge}^3 (V_n^*)$  (the cosets are represented by polynomials $x_ix_jx_k$, where $i,j,k=1,\dots,n$ are distinct).
\end{proof}

For $\omega\in H^3(V_n ,\, k^\times)$ let 
\begin{equation}
\label{omega alt 2}
\omega_{alt}=\pi(\omega)\in {\bigwedge}^3 (V_n^*).
\end{equation}

\begin{proposition}
\label{vinimageFeven}
Let $v\in V_n$ and let $\omega\in H^3(V_n,\, k^\times)$. 
Then $v$, regarded as a simple object in $\C(V_n,\, \omega)$,  is
in the image of  $\Inv(\Z(\C(V_n,\, \omega))) \to  \Inv(\C(V_n,\, \omega))$ if and only if $v\in \Rad(\omega_{alt})$.
\end{proposition}
\begin{proof}
This is similar to proof of Proposition~\ref{vinimageFodd}.
\end{proof}

\begin{corollary}
\label{ptd iff 2}
The category $\Z(\C(V_n,\, \omega))$ is pointed if and only if $\omega_{alt}=0$.
\end{corollary}

\begin{corollary}
\label{exact seq for BrPic in nondeg case p=2}
Corollary~\ref{walt nondeg} and Theorem~\ref{exact seq for BrPic in nondeg case}
hold for $p=2$ (but note that the meaning of $\omega_{alt}$ is different in this case, cf.\ \eqref{omega alt 2})
\end{corollary}

\begin{remark}
Representation categories of twisted group doubles of elementary abelian $2$-groups and  braided tensor equivalences
between them were studied by Goff, Mason, and Ng in \cite{GMN}. The above results about the third cohomology of $V_n$ 
and invertible objects of $\Z(\C(V_n,\, \omega))$ can  be derived from that paper. 
\end{remark}


\section{Computing  Brauer-Picard groups: examples}
\label{section 6}

Let $p$ be a prime. A $p$-group $G$ is called {\em extra special} if its center $Z$ is cyclic of order $p$ 
and $G/Z$ is elementary abelian. Such groups are well known:  for each positive integer $n$  there exist 
precisely two non-isomorphic extra special $p$-groups of order $p^{2n+1}$. 
 
Since $G$ is a central extension of the form
 \begin{equation}
 \label{Fp V2n}
 0 \to \mathbb{F}_p \to G \to V_{2n}\to 0,
 \end{equation}
corresponding to some cohomology class 
$\kappa_G\in H^2(V_n,\, \mathbb{F}_p)$
we can apply the results of Section~\ref{sect Deepak}. Namely, there is a braided equivalence 
\begin{equation}
\label{omega extra}
\Z(\C(G,\,1)) \cong \Z(\C(V_{2n+1},\,\omega_G)),
 \end{equation}
where $\omega_G \in H^3(G,\,k^\times)$ is obtained from $\kappa_G$ as follows. Choose a generator $x_0$
of $H^1(\mathbb{F}_p,\, \mathbb{F}_p) = \mathbb{F}_p^*$ and consider the cup product
\[
H^1(\mathbb{F}_p,\, \mathbb{F}_p) \ot_{\mathbb{F}_p} H^2(V_n,\, \mathbb{F}_p) \to H^3(V_{2n+1},\,\mathbb{F}_p)
: x_0\ot \kappa \mapsto x_0\cup \kappa. 
\]

\begin{proposition}
\label{from kappa to omega}
$\omega_G$ is the image of $x_0\cup \kappa_G$ under the projection
\[
H^3(V_{2n+1},\,\mathbb{F}_p)\to   H^3(V_{2n+1},\,k^\times).
\] 
\end{proposition}
\begin{proof}
This follows from \eqref{kappa rho}.
\end{proof}

\subsection{Brauer-Picard groups of representation categories of extra special $p$-groups for odd $p$}
\label{BrPic extra odd}

Let $p$ be an odd prime and let $n$ be a positive integer.  
Let $D$ and $Q$ denote extra special groups of order $p^{2n+1}$ and exponents $p$ and $p^2$
respectively.  They can be constructed as follows. Let $M= (\mathbb{Z}/p\mathbb{Z} \times \mathbb{Z}/p\mathbb{Z} ) \rtimes \mathbb{Z}/p\mathbb{Z}$
and $N = \mathbb{Z}/p^2\mathbb{Z}  \rtimes \mathbb{Z}/p\mathbb{Z}$ be non-abelian groups of order $p^3$. Then 
$D$ is the central product of $n$ copies of $M$ and $Q$ is  the central product of $N$ and $n-1$ copies of $M$. 

It is straightforward to compute cohomology classes  $\kappa_D,\,\kappa_Q\in H^2(V_{2n},\, \mathbb{F}_p)$ 
corresponding to  extension \eqref{Fp V2n} with $G=D,\, Q$. We have
\begin{eqnarray}
\kappa_D &=& \sum_{i=1}^n \, x_{2i-1}\,x_{2i}, \label{kappaD}\\
\kappa_Q &=& \sum_{i=1}^n \, x_{2i-1}\,x_{2i} +y_1, \label{kappaQ}
\end{eqnarray}
where generators $x_i,\,y_i$ are defined by \eqref{x i} and \eqref{y i}. 

Recall from Proposition~\ref{H3 when p odd} that there is a $GL_{2n+1}(\mathbb{F}_p)$-module isomorphism
\begin{equation}
\label{H3 now odd}
H^3(V_{2n+1} ,\, k^\times) \cong {\bigwedge}^3 (x_0,\, x_1,\dots, x_{2n}) \oplus \Sym^2(z_0,\, z_1,\dots, z_{2n}),
\end{equation}
hence elements $\omega \in H^3(V_{2n+1} ,\, k^\times)$   correspond to pairs $(\omega_{alt},\, \omega_{sym})$,
where $\omega_{alt}$ is a degree $3$ element of the exterior algebra and 
 $\omega_{sym}$ is a degree $2$ element of the symmetric algebra.

Below we identify $\omega$ with its image under isomorphism \eqref{H3 now odd}.

\begin{proposition}
\label{omegas D,Q}
We have
\begin{eqnarray}
\label{omegaD}
\omega_D &=&  \left( x_0 \wedge \left( \sum_{i=1}^n \, x_{2i-1}\wedge x_{2i}, \right),\, 0 \right), \\
\label{omegaQ}
\omega_Q &=&  \left( x_0  \wedge \left( \sum_{i=1}^n \, x_{2i-1}\wedge x_{2i}, \right),\,  z_0 z_1 \right). 
\end{eqnarray} 
\end{proposition}
\begin{proof}
This follows from equations \eqref{kappaD}, \eqref{kappaQ}, and Proposition~\ref{from kappa to omega}.
\end{proof}

It follows from Theorem~\ref{exact seq for BrPic in nondeg case} that  the  Brauer-Picard groups $\BrPic(\C(D,\,1))$
and $\BrPic(\C(Q,\,1))$ are extensions of $\Stab(\omega_D)$ and $\Stab(\omega_Q)$,  respectively, by elementary abelian
$p$-groups.  One can find these stabilizers  using the explicit formulas \eqref{omegaD} and \eqref{omegaQ}.
Below we do it for $\omega_D$. 

Let $V$ be a finite dimensional vector space. For any subgroup $G\subset GL(V)$ define the corresponding
affine group $AffG = V \rtimes G$. 

\begin{proposition}
\label{Stab D}
Suppose $n>1$. Then 
\[
\Stab(\omega_{D}) \cong AffSp_{2n}(\mathbb{F}_p) \rtimes \mathbb{F}_p^\times.
\]	
Here $Sp_{2n}(\mathbb{F}_p)$ denotes the symplectic group. 
\end{proposition}
\begin{proof}
Let $\{e_0,\, e_1,\dots, e_{2n}\}$ be the standard basis of $V_{2n+1}$ so that
$\{x_0,\, x_1,\dots, x_{2n}\}$ is the dual basis of $V_{2n+1}^*$.  By the rank of $\sum_i a_i\wedge b_i\in \bigwedge^2\, V^*$
we mean the rank of the associated linear endomorphism of $v$ given by
\[
V \to V: v \mapsto \sum_i\, \langle v,\, a_i\rangle b_i - \langle v,\, b_i\rangle a_i.
\]
For every $g\in \Stab(\omega_D)$
and every $v\in V_{2n+1}$ the ranks of $\iota_{g(v)}(\omega_D)$ and $\iota_{v}(\omega_D)$ are equal. 
It follows that  the span of $\{e_1,\dots, e_{2n}\}$ is stable
under $\Stab(\omega_D)$.  Indeed, non-zero vectors $v$ lying in this span are characterized by the property that $\iota_v(\omega_D)$
has rank $2$.  Let $s= \sum_{i=1}^n \, x_{2i-1}\wedge x_{2i}$ (it corresponds to the symplectic form).
An element $g\in \Stab(\omega_D)$ must map $e_0$  to $\lambda e_0 + \sum_{i=1}^{2n} \,v_i e_i$ and the dual of
the  restriction of $g$ on the span of $\{e_1,\dots, e_{2n}\}$ must map $s$ to $\lambda^{-1} s$, where $v= (v_1,\dots v_{2n})$ 
is an arbitrary vector in $\mathbb{F}_p^{2n}$ and $\lambda\in \mathbb{F}_p^\times$. 

Thus,  $\Stab(\omega_{D})$ is generated by  matrices  of the form 
	\[ 
	 \left[
    \begin{array}{r@{}c|c@{}l}
  &	1 & 0 \\ \hline
  &    v &  
	\begin{smallmatrix} 
	  &&&&&& \\	&&&&&& \\
	  &&&&&& \\	&&&&&& \\  
      &&&	\huge{{M}} &&& \\
	  &&&&&& \\	&&&&&& \\      
      &&&&&& \\	&&&&&&
    \end{smallmatrix}
          &
    \end{array} 
\right]
	\qquad \text{ and } \qquad
	\left[
    \begin{array}{r@{}c|c@{}l}
  &	\lambda & 0 \\ \hline
  &    0 &  
       \begin{smallmatrix}\rule{0pt}{2ex}
        \left(\begin{smallmatrix}\rule{0pt}{2ex}
        \lambda^{-1} & 0 \\
        0	&	1
      \end{smallmatrix} \right)
         & & 0 \\
          &\ddots&\\
        0 & & 
        \left(\begin{smallmatrix}\rule{0pt}{2ex}
        \lambda^{-1} & 0 \\
        0	&	1
      \end{smallmatrix} \right)
      \end{smallmatrix}    &
    \end{array} 
\right],
	\]
where $v \in \mathbb{F}_p^{2n}$, $M \in Sp_{2n}(\mathbb{F}_p)$, and $\lambda \in \mathbb{F}_p^\times$.
This clearly implies the statement. 
\end{proof}

When $n=1$, i.e., when $D,\,Q$  are extra special groups of order $p^3$, there is a neat precise 
description of their Brauer-Picard groups.

\begin{proposition}
\label{Stab D,Q  n=1}
Let $n=1$. Then 
\[
\Stab(\omega_D)=SL_3(\mathbb{F}_p)\quad  \text{and} \quad   \Stab(\omega_Q)= AffO_2(\mathbb{F}_p).
\]
\end{proposition}
\begin{proof}
For any $A\in GL_3(\mathbb{F}_p)$ we have 
\[
A (x_0\wedge x_1\wedge x_2) = \det(A)(x_0\wedge x_1\wedge x_2)
\]
in $\bigwedge^3(x_0,\, x_1,\, x_2)$, which proves the first equality. The group of matrices 
in $SL_3(\mathbb{F}_p)$ stabilizing  $z_0 z_1\in \Sym^2(z_0,\, z_1,\, z_2)$ 
is precisely the $2$-dimensional affine orthogonal group, which proves the second equality.
\end{proof}

\begin{corollary}
\label{BrPic D, Q n=1}
Let $D,\,Q$ be the extra special groups of order $p^3$ of exponents $p$ and $p^2$, respectively.
Then 
\begin{equation}
\label{Brianna's BrPic}
\BrPic(\C(D,\,1)) \cong SL_3(\mathbb{F}_p) \quad \text{ and } \quad \BrPic(\C(Q,\,1)) \cong AffO_2(\mathbb{F}_p).
\end{equation}
\end{corollary}
\begin{proof}
Note that the homomorphisms $\iota(\omega_{alt})$ for $\omega= \omega_D,\, \omega_Q$ 
in \eqref{exact sequence for Vn} are isomorphisms, so $\BrPic(\C(V_{3},\,\omega)) =\Stab(\omega)$
and the isomorphisms follow from Proposition~\ref{Stab D,Q  n=1}.
\end{proof}

\begin{remark}
The first of isomorphisms in \eqref{Brianna's BrPic} was established by  Riepel in \cite{R}
by different methods.
\end{remark}

\subsection{Brauer-Picard groups of representation categories of extra special $2$-groups}
\label{BrPic extra 2}

Let $D$ and $Q$ denote extra special groups of order $2^{2n+1}$.  They can be constructed as follows.  
The group $D$ is the central product of $n$ copies of the dihedral group of order $8$  and $Q$ 
is the central product $n-1$ copies of the dihedral group of order $8$ and one copy of the quaternion group.

It is straightforward to compute cohomology classes  $\kappa_D,\,\kappa_Q\in H^2(V_{2n},\, \mathbb{F}_2)$ 
corresponding to  extension \eqref{Fp V2n} with $G=D,\, Q$. We have
\begin{eqnarray}
\kappa_D &=& \sum_{i=1}^n \, x_{2i-1}\,x_{2i}, \label{kappaD p=2}\\
\kappa_Q &=& \sum_{i=1}^n \, x_{2i-1}\,x_{2i} \,+\, x_1^2 \,+\, x_2^2, \label{kappaQ p=2}
\end{eqnarray}
where generators $x_i$ are defined by \eqref{x i when p=2}. 

Recall from Proposition~\ref{H3 when p=2} that there is an isomorphism of $GL_{2n+1}(\mathbb{F}_2)$-modules:
\begin{equation}
\label{H3 now even}
H^3(V_{2n+1} ,\, k^\times) \cong \Sym^3 (x_0, x_1,\dots, x_{2n}) / \mathbb{F}_2 \langle x_i^2x_j+x_ix_j^2 \mid \, i,j=0,\dots,2n\rangle.
\end{equation}

Below we identify $\omega$ with its image under \eqref{H3 now even}.

\begin{proposition}
\label{omegas D,Q p=2}
We have
\begin{eqnarray}
\omega_D &=& \sum_{i=1}^n \, x_0\,x_{2i-1}\,x_{2i}, \label{omegaD p=2}\\
\omega_Q &=& \sum_{i=1}^n \, x_0\,x_{2i-1}\,x_{2i} \,+\, x_0\,x_1^2 \,+\, x_0\,x_2^2. \label{omegaQ p=2}
\end{eqnarray} 
\end{proposition}
\begin{proof}
This follows from equations \eqref{kappaD p=2}, \eqref{kappaQ p=2}, and Proposition~\ref{from kappa to omega}.
\end{proof}

Recall from Proposition~\ref{H3 when p even} that there is a short exact sequence of $GL_{2n+1}(\mathbb{F}_2)$-modules
\begin{equation}
0\to \Sym^2(x_0,\, x_1,\dots, x_{2n}) \to H^3(V_{2n+1} ,\, k^\times)  \xrightarrow{\pi} {\bigwedge}^3 (x_0,\, x_1,\dots, x_{2n}) \to 0.
\end{equation}
\begin{proposition}
We have $ \pi(\omega_D) =  \pi(\omega_Q)= \omega_{alt}$, where
\begin{equation}
\label{omega alt p=2}
\omega_{alt} =  x_0 \wedge \left( \sum_{i=1}^n \, x_{2i-1}\wedge x_{2i}, \right).
\end{equation}
\end{proposition}
\begin{proof}
The homomorphism $\pi:H^3(V_{2n+1} ,\, k^\times)  \to {\bigwedge}^3(V_{2n+1}^*)$ is described explicitly in the proof of 
Proposition~\ref{H3 when p even}. The result is immediate from there.
\end{proof}

In view of Corollary~\ref{exact seq for BrPic in nondeg case p=2}  the  Brauer-Picard groups of $\C(D,\,1)$
and $\C(Q,\,1)$ are extensions of $\Stab(\omega_D)$ and $\Stab(\omega_Q)$,  respectively, by elementary abelian
$2$-groups.  The above stabilizers can be found using the explicit formulas \eqref{omegaD p=2} and \eqref{omegaQ p=2}.

We consider $\Stab(\omega_D)$ and $\Stab(\omega_Q)$ as subgroups of $\Stab(\omega_{alt})$.

\begin{proposition}
\label{Stab alt p=2}
Suppose $n>1$. Then 
\[
\Stab(\omega_{alt}) \cong AffSp_{2n}(\mathbb{F}_2).
\]	
\end{proposition}
\begin{proof}
This is similar to the proof of Proposition~\ref{Stab D}.
\end{proof}

\begin{proposition}
\label{inclusion}
Let $n>1$. The projection from $AffSp_{2n}(\mathbb{F}_2)$ to $Sp_{2n}(\mathbb{F}_2)$ induces inclusions 
\[
\Stab(\omega_D) \hookrightarrow Sp_{2n}(\mathbb{F}_2)  \quad  \text{ and }  \quad  \Stab(\omega_Q) \hookrightarrow Sp_{2n}(\mathbb{F}_2)
\]
\end{proposition}
\begin{proof}
A simple computation verifies that the kernel of this projection intersects both $\Stab(\omega_D)$ and $\Stab(\omega_Q)$ trivially.
\end{proof}

\begin{proposition}
Let $n>1$. Then
\[
\Stab(\omega_D) \cong Sp_{2n}(\mathbb{F}_2)  \quad  \text{ and }  \quad  \Stab(\omega_Q) \cong Sp_{2n}(\mathbb{F}_2)
\]
\end{proposition}
\begin{proof}
Second cohomology classes $\kappa_D$ and $\kappa_Q$ defined in equations \eqref{kappaD p=2} and \eqref{kappaQ p=2}, represent the two equivalence classes of quadratic forms in even dimension \cite[3.4.7]{W}.
Their respective stabilizers, the orthogonal groups $O^+_{2n}(\mathbb{F}_2)$ and $O^-_{2n}(\mathbb{F}_2)$, 
are identified with subgroups  of $\Stab(\omega_D)$ and $\Stab(\omega_Q)$.

It is known that  $O^+_{2n}(\mathbb{F}_2)$ and $O^-_{2n}(\mathbb{F}_2)$ are maximal subgroups of $Sp_{2n}(\mathbb{F}_2)$ \cite[Theorem 1.5]{P}. 

Let $M \in  GL(V_{2n+1})$ be defined by
\[
M(v_0,\, v_1,\, v_2,\, v_3,\, v_4, \dots, v_{2n-1},\, v_{2n}) = (v_0,\, v_1,\, v_2,\, v_0+v_3+v_4,\, v_4, \dots, v_{2n-1},\, v_{2n}).
\]
Note that $M \in \Stab(\omega_D)$ and $M \in \Stab(\omega_Q)$, however, it is not an element of either orthogonal subgroup. 
Thus the images of inclusions specified in Proposition~\ref{inclusion} properly contain maximal subgroups.
It follows that these are isomorphisms.
\end{proof}

\begin{proposition}
\label{Stab D,Q  n=1 p=2}
Let $n=1$. Then 
\[
\Stab(\omega_D)\cong S_4\quad  \text{and} \quad   \Stab(\omega_Q)= S_3.
\]
\end{proposition}
\begin{proof}
For any $\omega\in H^3(V_3,\, k^\times)= \Sym^3(x_0,\, x_1,\, x_2)/\mathbb{F}_2 \langle x_i^2x_j+x_ix_j^2 \mid \, i,j=0,1,2\rangle$
the evaluation map
\[
V_3 \to \mathbb{F}_2 : (v_0,\, v_1,\, v_2) \mapsto \omega(v_0,\, v_1,\,v_2)
\]
is well defined and, furthermore,  $\omega$ is completely determined by the set of vectors of $V_3$ mapped to $1$ 
(cf.\  \cite{M}). 

For $\omega= \omega_D=x_0x_1x_2$ this set consists of the single vector $(1,\, 1,\,1)$. So $\Stab(\omega_D)$
is precisely the subgroup of automorphisms of $V_3$ fixing this vector, i.e.,  
$\Stab(\omega_D)\cong AffGL_2(\mathbb{F}_2)\cong \mathbb{F}_2^2 \rtimes GL_2(\mathbb{F}_2)\cong  S_4$.

For $\omega=\omega_Q= x_0x_1x_2+ x_0x_1^2+x_0x_2^2$ this set consists of vectors $(1,\, 1,\,1),\, (1,\, 1,\,0)$,
and $(1,\, 0,\,1)$.
The group $\Stab(\omega_Q)$ consists of automorphisms of $V_3$ permuting these vectors. 
Since they form a basis of $V_3$,  we conclude $\Stab(\omega_Q) \cong S_3$.
\end{proof}

\begin{corollary}
\label{BrPic D, Q n=1 p=2}
Let $D,\,Q$ be the dihedral group and quaternion groups of order $8$, respectively.
Then 
\begin{equation}
\label{Brianna's BrPic p=2}
\BrPic(\C(D,\,1)) \cong S_4 \quad \text{ and } \quad \BrPic(\C(Q,\,1)) \cong S_3.
\end{equation}
\end{corollary}
\begin{proof}
Note that the homomorphisms $\iota(\omega_{alt})$ for $\omega= \omega_D,\, \omega_Q$ 
in \eqref{exact sequence for Vn} are isomorphisms, so $\BrPic(\C(V_{3},\,\omega)) =\Stab(\omega)$
and the isomorphisms follow from Proposition~\ref{Stab D,Q  n=1 p=2}.
\end{proof}

\begin{remark}
The isomorphisms in \eqref{Brianna's BrPic p=2} were established in \cite{NR}
by different methods.
\end{remark}

\subsection{Pointed $p$-categories coming from metric modular Lie algebras}

In view of isomorphism \eqref{H3Ep=odd} one can produce interesting examples of $3$-cocycles on elementary abelian groups
as follows.

Let $F$ be a finite field of  characteristic $p>3$. Below we consider   finite dimensional Lie algebras over $F$.
We refer the reader to \cite{S} for the theory of modular Lie algebras. 

\begin{definition}
A {\em pre-metric Lie algebra} is a Lie algebra $\g$ equipped with 
an invariant symmetric bilinear form $(\, , \,)$, i.e., such that
\[
([a,\,b],\, c) = (a,\, [b,\,c]),
\]
for all $a,b,c\in \mathfrak{g}$.  A {\em metric Lie algebra} is a pre-metric Lie algebra such that $(\, , \,)$ is non-degenerate.
\end{definition}

For a Lie algebra $\g$ let $\Aut(\g)$ denote the group of Lie algebra automorphisms of $\g$.
For a pre-metric Lie algebra $\g$ let $\Aut_m(\g)\subset \Aut(\g)$ denote the group of Lie algebra automorphisms of $\g$
preserving $(\, , \,)$.

Consider the following bilinear symmetric and trilinear alternating forms on $\mathfrak{g}$:
\[
\tilde\omega_{sym}(a,\,b) = (a,\,b) 
\quad  \text{  and } \quad 
\tilde\omega_{alt}(a,\,b,\,c)=([a,\,b],\, c) , \qquad a,b,c\in \mathfrak{g},
\]
and identify them with $\omega_{sym}\in \Sym^2(\g^*)$ and
$\omega_{alt}\in \bigwedge^3(\g^*)$ by means of symmetrization and anti-symmetrization maps.

Let $V$ denote the underlying additive group of $\mathfrak{g}$. It is an elementary abelian $p$-group.

Set  
\begin{equation}
\label{omega for Lie}
\omega = \omega_{alt}+ \omega_{sym} \in H^3(V,\, k^\times).
\end{equation}

\begin{proposition}
\label{stab =autm}
Let $\g$ be a metric Lie algebra such that $\mathfrak{g} = [\mathfrak{g},\,\mathfrak{g}]$
and let $\omega\in H^3(V,\, k^\times)$ be the $3$-cocycle constructed above.  Then
$\Stab(\omega)\cong \Aut_m(\g)$.
\end{proposition} 
\begin{proof}
It is clear that each $\phi\in \Aut_m(\g)$ stabilizes $\tilde\omega_{sym}$ and $\tilde\omega_{alt}$.
Hence, it stabilizes $\omega$.  Conversely, if $\psi$
is a group  automorphism of $V$ stabilizing $\omega$ then it must stabilize both $\omega_{sym}$
and $\omega_{alt}$.  The former condition means that $\psi$ preserves  $(\, , \,)$ and the latter one means that it is a Lie algebra
homomorphism.
\end{proof}

\begin{proposition}
\label{w alt for metric}
Let $\g$ be a pre-metric Lie algebra. Then  $\omega_{alt}\in \bigwedge^3(\g^*)$ is non-degenerate 
if and only if $\mathfrak{g} = [\mathfrak{g},\,\mathfrak{g}]$
and $(\, , \,)$  is non-degenerate (i.e., $\g$ is a metric Lie algebra).
\end{proposition}
\begin{proof}
Non-degeneracy of $\omega_{alt}$ is equivalent to non-degeneracy of the alternating trilinear form  
$\tilde\omega_{alt}$. This implies the statement.
 \end{proof}

\begin{corollary}
\label{when g metric}
Suppose that $\g$ is a metric Lie algebra such that $\mathfrak{g} = [\mathfrak{g},\,\mathfrak{g}]$. 
There is an exact sequence
\begin{equation}
\label{Autm exact}
0 \to V \to{\bigwedge}^2 (V^*) \to \BrPic(\C(V,\, \omega)) \to \Aut_m(\g) \to 0.
\end{equation}
\end{corollary}
\begin{proof}
This follows from Theorem~\ref{exact seq for BrPic in nondeg case}
and Propositions~\ref{stab =autm} and \ref{w alt for metric}.  
\end{proof}

One can take $(\, , \,)$ to be the Killing form of $\mathfrak{g}$ :
\[
(a,\,b) = \Tr_\mathfrak{g} (\ad(a)\,\ad(b)),\qquad a,b\in \mathfrak{g},
\]
where $\ad$ denotes the adjoint representation of $\mathfrak{g}$. Let $\omega^\g\in H^3(V,\, k^\times)$
denote the corresponding $3$rd cohomology class defined by \eqref{omega for Lie}.

\begin{remark}
If $\g$  has a non-degenerate Killing form then $\g$ is a direct sum of simple Lie algebras.
Unlike for Lie algebras  over $\mathbb{C}$,  the converse to this statement is false 
in positive characteristic. All simple classical Lie algebras with a non-degenerate Killing over 
a field $F$ with $\text{char}(F)>3$ are known. The necessary and sufficient condition is that  
$\text{char}(F)$  should not divide the determinant of the Killing form of the corresponding simple 
complex Lie algebra, see \cite[Chapter II \S 9]{S}.

It can even happen that every trace form of a simple $\g$ is degenerate, this is the case, e.g., for $\g =
\mathfrak{s}\mathfrak{l}_n(F)$ when $\text{char}(F)$ divides $n$.
\end{remark}

\begin{proposition}
Let $\g$ be a simple Lie algebra with a non-degenerate Killing form. We have an exact sequence
\[
0 \to V \to {\bigwedge}^2(V^*) \to \BrPic(\C(V,\, \omega^\g)) \to \Aut(\g) \to 0,
\]
\end{proposition}
 \begin{proof}
 It is clear that when $(\, , \,)$ given by the Killing form of $\g$ we have $\Aut_m(\g)=\Aut(\g)$.
The result follows from Corollary~\ref{when g metric}.
 \end{proof}
 
 \begin{remark}
 The automorphism groups of classical simple modular Lie algebras $\g$ are known \cite{S}.
 
 Thus, finite simple groups of Lie type naturally appear as composition factors of  groups of autoequivalences 
 and Brauer-Picard groups of pointed fusion categories.
 \end{remark}
 
 \begin{example}
 Take $\g = \mathfrak{s}\mathfrak{l}_2(F)$. Since $p$ is odd, the Killing form is non-degenerate.
 Let
 \[
 e=\left[\begin{array}{cc} 0 & 1 \\ 0 & 0 \end{array}\right],
 \quad f=\left[\begin{array}{cc} 0 & 0 \\ 1 & 0 \end{array}\right],
 \quad \text{and}
 \quad h=\left[\begin{array}{cc} 1 & 0 \\ 0 & -1 \end{array}\right],
 \]
 be a basis  of $\mathfrak{g}$ so that
 \[
 ad(e)=\left[\begin{array}{ccc} 
 0 & 0 & -2 \\ 0 & 0 & 0 \\ 0 & 1 & 0 
 \end{array}\right],
 \; ad(f)=\left[\begin{array}{ccc} 
 0 & 0 & 0 \\ 0 & 0 & 2 \\ -1 & 0 & 0
 \end{array}\right],
 \; \text{and}
 \quad ad(h)=\left[\begin{array}{ccc} 
 2 & 0 & 0 \\ 0 & -2 & 0 \\ 0 & 0 & 0
 \end{array}\right].
 \]
 The corresponding $3$-cocycle $\omega^{\g}$ is given by
 \[
 \omega^\g =  \left( 8 x_e  \wedge  x_f \wedge x_h, \,  4z_e  z_f \, + \, 8 z_h^2 \right). 
 \]
 In particular, when $F=\mathbb{F}_p$ we have $\BrPic(\C(V,\, \omega^\g)) =SO_3(\mathbb{F}_p)$.
 
 \end{example}

\subsection{Representation category of the Kac-Paljutkin Hopf algebra}

Recall that the {\em Kac-Paljutkin Hopf algebra} $KP$ is the unique non-commutative and non-cocommutative 
semisimple Hopf algebra of dimension $8$ \cite{KP}.

Let $\KP$ be the fusion category of representations of $KP$. Below we compute the Brauer-Picard 
group of $\KP$.

Note that $\KP$ is a particular case of a {\em  Tambara-Yamagami} fusion category \cite{TY}. Recall that
such categories $\TY(A,\,\chi,\,\tau)$ are defined as follows. Let $A$ be a finite  abelian group.
The set of simple objects of $\TY(A,\,\chi,\,\tau)$
is $\{ a\mid a\in A\} \cup \{m\}$ with tensor products given by 
\[
a\ot b =ab,\qquad a\ot m =m \ot a = m,\qquad m\ot m =\bigoplus_{a\in A},
\]
and associativity constraints determined by a non-degenerate symmetric bilinear form $\chi:A\ot A\to k^\times$
and $\tau\in k$ such that $\tau^2= |A|^{-1}$ (see \cite{TY} for details of the definition).
It was shown in \cite[Section 4]{TY}  that 
\begin{equation}
\label{KP=TY}
\KP\cong \TY(\mathbb{Z}_2 \times \mathbb{Z}_2,\, \chi,\, \frac{1}{2}),
\end{equation}  
where $\chi$  is the bilinear form on  $\mathbb{Z}_2 \times \mathbb{Z}_2 =\{1,\,a,\, b,\, c\}$ determined by
\[
\chi(a,\,a)=1,\quad\chi(b,\,b)=\chi(c,\,c)=-1,\quad  \chi (b,\,c)=1.
\]
 
The centers of Tambara-Yamagami categories and their $S$- and $T$-matrices were explicitly 
described in \cite[Section 4C]{GNN}.  The following result is immediate from that description.

\begin{lemma}
\label{ZTY inv}
The forgetful homomorphism 
\[
\Inv(\Z(\TY(A,\,\chi,\,\tau)))\to \Inv(\TY(A,\,\chi,\,\tau))
\]
is surjective.
\end{lemma}

Let $D$ and $Q$  denote, respectively,   the dihedral and quaternion groups of order~$8$.
We will use the following presentation of $D$ by generators and relations:
\[
D =\langle r,\, s \mid r^4=s^2=1,\ rs=sr^{-1} \rangle.
\]

\begin{lemma}
\label{subcategories of ZKP}
\begin{enumerate}
\item[(i)]  $\Z(\KP)_{pt}\cong \Rep(\mathbb{Z}_2^3)$ is a symmetric non-Tannakian category. 
\item[(ii)] $\Z(\KP)$ has precisely two Lagrangian subcategories; both of them are equivalent to $\Rep(D)$.
\item[(iii)] $\KP$ has a unique, up to equivalence, fiber functor (i.e., a tensor functor to $\Vec$).
\end{enumerate}
\end{lemma} 
\begin{proof}

Using the description of $\Z(\KP)$ from \cite{GNN} and equivalence \eqref{KP=TY}   one sees that every invertible object $x$ of $\KP$
has precisely $2$ structures of an object  of $\Z(\KP)$ and the square of the braiding between objects supported 
on  invertible objects $x,\, y\in A$  is $\chi(x,\,y)\id_{xy}$. This proves (i).

Inspecting the  $S$- and $T$-matrices of $\Z(\TY(A,\,\chi,\,\tau))$ we conclude that Lagrangian subcategories of $\Z(\KP)$ 
are precisely  those that contain $4$ invertible objects supported on $1,\,a$ and one two-dimensional simple object
supported either on $1\oplus a$ or on $b\oplus c$.  So there are two such subcategories. Since the dihedral group $D$ 
and  the quaternionic group $Q$ are the only non-abelian groups of order $8$ the above Lagrangian categories
must be equivalent to either $\Rep(D)$ or  $\Rep(Q)$. The latter case leads to a contradiction: 
it implies that $\Z(\KP) =\Z(\C(Q,\,\eta))$  for some non-trivial $\eta \in H^3(Q,\, k^\times)$.  
But $\C(Q,\,\eta)$ is not categorically Morita equivalent to the representation category of  a Hopf algebra by \cite{O}.
This proves (ii).

Finally, (iii) follows for the classification of fiber functors on Tambara-Yamagami categories \cite{T}. 
\end{proof}

\begin{lemma}
\label{Morita equivalence class of KP}
Let $\C$ be a fusion category categorically Morita equivalent to $\KP$. Then either $\C \cong \KP$
or $ \C\cong \C(D,\, \omega)$ for some non-trivial $3$-cocycle $\omega\in H^3(D,\, k^\times)$.
\end{lemma}
\begin{proof}
If $\C$ is not pointed then it must be a Tambara-Yamagami category. Combining Lemmas~\ref{ZTY inv} 
and~\ref{subcategories of ZKP}(i)  we conclude that $\Inv(\C) = \mathbb{Z}_2 \times \mathbb{Z}_2$.
It is proved in \cite[Section 5.2]{NN} that  the only fusion category with the group of invertible objects $\mathbb{Z}_2 \times \mathbb{Z}_2$
categorically Morita equivalent to $\KP$ is $\KP$ itself.

If $\C$ is pointed then it corresponds to a Lagrangian subcategory in $\Z(\KP)$ and the statement follows from
Lemma~\ref{subcategories of ZKP}(ii). 
\end{proof}

Using Lemma~\ref{subcategories of ZKP}(iii), description of $\omega$ from \cite{KP}, 
and classification of fiber functors on group-theoretical categories \cite{O}
one can list all  indecomposable  $\C(D,\,\omega)$-module categories
(and, hence, all indecomposable $\KP$-module categories). There are precisely $6$ such categories, parameterized 
by conjugacy classes of subgroups of $D$ on which $\omega$ has trivial restriction, namely:
$1,\, \langle r^2\rangle,\, \langle r^2,\, s\rangle,\,  \langle sr\rangle,\, \langle r^2,\, sr\rangle,\,  \langle s\rangle$.

\begin{remark}
The same list of $\KP$-module categories can be obtained using the classification of module categories
over Tambara-Yamagami categories from \cite{MM}. We are grateful to Ehud Meir for sharing his calculations with us. 
\end{remark}

\begin{corollary}
\label{2 and 4}
There are precisely $2$ equivalence classes of indecomposable $\KP$-module categories such that the dual fusion category is pointed
and precisely $4$ equivalence classes of indecomposable $\KP$-module categories such that the dual fusion category is equivalent to $\KP$.
\end{corollary}
\begin{proof}
Module $\KP$-categories with pointed duals correspond to Lagrangian subcategories of $\Z(\KP)$. Since there are $6$
classes of indecomposable $\KP$-module categories, the result follows from Lemma~\ref{subcategories of ZKP}(ii)
\end{proof}

\begin{corollary}
\label{order is 8}
$|\BrPic(\KP)|=8$.
\end{corollary}
\begin{proof}
By Corollary~\ref{2 and 4} $\KP$ has $4$ module categories with the dual $\KP$, so the statement 
follows from Proposition~\ref{orbits}  and the fact that $\Aut(\KP)=\mathbb{Z}_2$ \cite{T}. 
\end{proof}

\begin{corollary}
\label{Z2+Z2}
The image of the induction homomorphism
\begin{equation}
\label{KP ind}
\Aut(\C(D,\, \omega)) \to \Aut^{br}(\Z(\C(D,\, \omega))) =\Aut^{br}(\Z(\KP))
\end{equation}
is isomorphic to $\mathbb{Z}_2\times \mathbb{Z}_2$.
\end{corollary}
\begin{proof}
By Remark~\ref{transitivity on L} (see also \cite{NR}) $\Aut^{br}(\Z(\KP))$ acts transitively on Lagrangian subcategories of $\Z(\KP)$
and the image of $\eqref{KP ind}$ is precisely the stabilizer of a Lagrangian subcategory $\Rep(D)$.
Hence, this image must contain $4$ elements   by Corollary~\ref{order is 8}. These elements are
the autoequivalences of the form $\ind(F_{a,\,\mu})$, where $a\in \Stab(\omega)$  and  $\mu$ is  a $2$-coboundary
such that 
\[
d^2\mu =\frac{\omega\circ (a\times a\times a)}{\omega}, 
\]
see Section~\ref{tens aut ptd}.  Note  that $\Out(D) \cong \mathbb{Z}_2$, where the class of non-trivial outer automorphism
is represented by $a\in \Aut(D)$ such that $a(r)=r,\, a(s)=sr$. 

Consider the action of the above autoequivalences $\ind(F_{a,\, \mu})$ on $\Z(\C(D,\, \omega))_{pt}$.
The invertible objects of the center are supported on $\{1,\,r^2\}$.  One easily checks that
autoequivalences $\ind(F_{1,\, \mu})$ fix central objects supported on $1$ and non-trivially permute those supported on $r^2$,
while autoequivalences $\ind(F_{a,\, \mu})$ non-trivially  permute central objects supported on $1$ and have order $2$. 
Thus, the induced autoequivalences form a group isomorphic to $\mathbb{Z}_2\times \mathbb{Z}_2$.
\end{proof}

\begin{proposition}
$\BrPic(\KP) \cong\mathbb{Z}_2\times \mathbb{Z}_2 \times \mathbb{Z}_2$.
\end{proposition}
\begin{proof}
Consider the Picard induction homomorphism (see Section~\ref{section 2.5})
\begin{equation}
\label{pic ind}
\Pic(\Z(\KP)_{pt}) \to \Pic(\Z(\KP)) = \BrPic(\KP) : \M \mapsto \Z(\KP) \bt_{\Z(\KP)_{pt}}  \M. 
\end{equation}
By Lemma~\ref{subcategories of ZKP}(i) and the result of \cite{C} we have $\Pic(\Z(\KP)_{pt}) = \mathbb{Z}_2^4$.
On the other hand, by \eqref{Ker PDC} the kernel of \eqref{pic ind} coincides with the universal grading group 
of $\Z(\KP)$, i.e., is isomorphic to $\mathbb{Z}_2^3$. Thus, the image of  \eqref{pic ind} has order $2$.  
Let $\alpha\in \Aut^{br}(\Z(\KP))$ correspond to the non-trivial element
of this image. Since $\alpha$  induces the trivial autoequivalence of $\Z(\KP)_{pt}$, it is not in the image of \eqref{KP ind},
cf.\ the proof of Corollary~\ref{Z2+Z2}.  Hence, $\alpha$  must permute the two Lagrangian subcategories of $\Z(\KP)$.  
Since these subcategories are fixed by the image of \eqref{KP ind} we get an injective  homomorphism from $\Aut^{br}(\KP)$ to
$\mathbb{Z}_2^3$ and the result follows from Corollary~\ref{order is 8}.
\end{proof}

\bibliographystyle{ams-alpha}

\begin{thebibliography}{A} 

\bibitem[A]{A} A.~Adem,
\textit{Cohomology of finite groups}, 2nd edition,
Springer, (2004).

\bibitem[BN]{BN} C.~Bontea, D.~Nikshych,
\textit{On the Brauer-Picard group of a finite symmetric tensor category},
J.\ Algebra, \textbf{440} (2015), 187--218. 

\bibitem[C]{C}  G.~Carnovale
{\em The Brauer group of modified supergroup algebras},
J.\ Algebra \textbf{305} (2006), no. 2, 993--1036.   



\bibitem[DPR]{DPR} R. Dijkgraaf, V. Pasquier, and P. Roche,
\textit{Quasi-Hopf algebras, group cohomology, and orbifold models},
Nuclear Phys.\ B, Proc. Suppl., \textbf{18},  (1991), no.~2, 60--72.

\bibitem[DGNO1]{DGNO1}  V.~Drinfeld, S.~Gelaki, D.~Nikshych, and V.~Ostrik,
Group-theoretical properties of nilpotent modular categories.
e-print {\em arXiv:0704.0195} (2007).

\bibitem [DGNO2]{DGNO2}  V.~Drinfeld, S.~Gelaki, D.~Nikshych, and V.~Ostrik,
\textit{On braided fusion categories I}, 
Selecta Mathematica,  \textbf{16}  (2010), no.\ 1,  1--119.

\bibitem [DMNO]{DMNO}  A.~Davydov, M.~M\"uger, D.~Nikshych, and V.~Ostrik,
{\em The Witt group of non-degenerate braided fusion categories},
Journal f\"ur die reine und angewandte Mathematik,
\textbf{677} (2013), 135--177.

\bibitem[EGNO]{EGNO} P.~Etingof, S.~Gelaki, D.~Nikshych, V.~Ostrik, 
{\em  Tensor categories}, 
Mathematical Surveys and Monographs,
\textbf{205}, American Mathematical Society (2015).


\bibitem[ENO1]{ENO1}   P.~Etingof, D.~Nikshych, and V.~Ostrik,
{\em  Fusion categories and homotopy theory}, 
Quantum Topology,    \textbf{1} (2010), no.\ 3,  209--273.

\bibitem[ENO2]{ENO2}   P.~Etingof, D.~Nikshych, and V.~Ostrik,
{\em  Weakly group-theoretical and solvable fusion categories}, 
Adv.\ Math, \textbf{226} (2011),  no.\ 1, 176--205

\bibitem[GP]{GP} C.~Galindo, J.~Plavnik,
{\em Tensor functor between Morita duals of fusion categories},
e-print {\em arXiv:1407.2783} (2014).

\bibitem[GNN]{GNN} S.~Gelaki, D.~Naidu, and D.~Nikshych,
{\em  Centers of graded fusion categories}, 
Algebra and Number Theory, \textbf{3} (2009), no.\ 8, 959--990. 

\bibitem[GN]{GN} S.~Gelaki, D.~Nikshych,
Nilpotent fusion categories. 
\textit{Advances in Mathematics}, \textbf{217} (2008), no. 3, 1053--1071.

\bibitem[GMN]{GMN} C.~Goff, G.~Mason, and S.-H.~Ng,
\textit{On the gauge equivalence of twisted quantum doubles of
elementary abelian and extra-special $2$-groups}.
J.\  Algebra, \textbf{312} (2007), no.\ 2, 849-875.


\bibitem[KP]{KP} G.~Kac, V.~Paljutkin.
\textit{Finite ring groups},
Trans.\ Moscow Math.\ Soc., (1966), 251--294.
 
 \bibitem[LP]{LP} S.~Lentner, J.~Priel,
 \textit{A decomposition of the Brauer-Picard group of the representation category of a finite group}, 
e-print {\em arXiv:1506.07832} (2015).

\bibitem[M]{M} G.~Mason,
\textit{Reed-Muller codes, the fourth cohomology group of a finite group, and the $\beta$-invariant},
J.\ Algebra, \textbf{312}  (2007), no.\~1, 218--227.

\bibitem[MM]{MM} E.~Meir, E.~Musicantov,
\textit{Module categories over graded fusion categories}, 
J.\ Pure and Applied Algebra, \textbf{216} (2012), no.~11,   2449--2466.

\bibitem[MN1]{MN1}  G.~Mason, S.-H.~Ng,
\textit{Group Cohomology and Gauge Equivalence of some Twisted Quantum Doubles},
Trans.\  Amer.\ Math.\ Soc., \textbf{353} (2001), 3465--3509. 

\bibitem[MN2]{MN2}  G.~Mason, S.-H.~Ng,
\textit{Cleft extensions and quotients of twisted quantum doubles},
in  Developments and retrospectives in Lie theory, 
Developments in Mathematics  \textbf{38} (2014), 229--246.

\bibitem [N]{N} D.~Naidu,
\textit{Categorical Morita equivalence for group-theoretical categories}, 
Comm.\ Algebra\ \textbf{35}, (2007), no. 11, 3544--3565.

\bibitem [NN]{NN} D.~Naidu, D.~Nikshych,
\textit{Lagrangian subcategories and braided tensor equivalences
of twisted quantum doubles of finite groups}, 
Comm.\ Math.\ Phys., \textbf{279} (2008), 845-872.

\bibitem [NNW]{NNW} D.~Naidu, D.~Nikshych, and S.~Witherspoon,
\textit{Fusion subcategories of representation categories of twisted quantum doubles of finite groups}, 
International Mathematics Research Notices (2009), no. 22, 4183--4219.

\bibitem[NR]{NR} D.~Nikshych, B.~Riepel,
\textit{Categorical Lagrangian Grassmannians and Brauer-Picard groups of pointed fusion categories},
J.\ Algebra, \textbf{411} (2014), 191--214.

\bibitem[O]{O} V.~Ostrik,
\textit{Module categories over the Drinfeld double of a finite group}. 
Int.\  Math.\ Res.\ Not.\ (2003), 1507--1520.

\bibitem[P]{P} H.~Pollatsek,  
\textit{First cohomology groups of some linear groups over fields of characteristic two},
Illinois J. \ Math., \textbf{15}  (1971), no. 3, 393--417

\bibitem[R]{R} B.~Riepel,
\textit{Brauer-Picard groups of pointed fusion categories}, 
Ph.D.\ Thesis, University of New Hampshire (2014).

\bibitem[S]{S}  G.B.~Seligman,
\textit{Modular Lie Algebras}, Springer, (1967). 

\bibitem[T]{T} D.~Tambara,
\textit{Representations of tensor categories with fusion rules of self-duality for abelian groups},
Israel J.\ Math., \textbf{118}  (2000), no. 1, 29--60.

\bibitem[TY]{TY} D.~Tambara, S.~Yamagami, 
\textit{Tensor categories with fusion rules of self-duality for finite abelian groups},
J. Algebra,  \textbf{209}  (1998),  no. 2, 692--707.

\bibitem[U]{U} B.~Uribe,
\textit{On the classification of pointed fusion categories up to weak Morita equivalence},
e-print {\em arXiv:1511.05522} (2015). 

\bibitem[W]{W} R.A.~Wilson,
\textit{The finite simple groups},
Springer, (2009).


\end{thebibliography}

\end{document}